\documentclass{mcom-l}

\usepackage{amsmath,amsfonts,amssymb,colortbl,float,graphicx}
\usepackage[numbers,sort&compress]{natbib}
\usepackage{enumerate,array,multirow,bbm}
\usepackage{booktabs,siunitx}

\newcommand{\bx}{{\bf x}}

\newcolumntype{K}[1]{>{\centering\arraybackslash}p{#1}}
\graphicspath{{Fig/}}

\theoremstyle{thmstyleone}%
\newtheorem{theorem}{Theorem}

\newtheorem{lemma}[theorem]{Lemma}
\newtheorem{remark}{Remark}%

\numberwithin{equation}{section}

\begin{document}

\title[A uniformly accurate MTI for nonlinear Klein-Gordon equation]{A uniformly accurate multiscale time integrator for the nonlinear Klein-Gordon equation in the nonrelativistic regime via simplified transmission conditions}

\author[W. Bao]{Weizhu Bao}
\address{Department of Mathematics, National University of Singapore, Singapore 117076, Singapore}

\email{matbaowz@nus.edu.sg}
\thanks{The work of the first author was partially supported by the Ministry of Education of Singapore
	under its AcRF Tier 1 funding A-8003584-00-00.}

\author[C. Liu]{Caoyi Liu}
\address{Department of Mathematics, National University of Singapore, Singapore 117076, Singapore}
\email{e1353374@u.nus.edu}
\thanks{}


\date{}

\dedicatory{}

\begin{abstract}
	We propose a new and simplified multiscale time integrator Fourier pseudospectral (MTI-FP) method for the nonlinear Klein-Gordon equation (NKGE) with a dimensionless parameter $\varepsilon\in(0,1]$ inversely proportional to the speed of light, and establish its uniform first-order accuracy in time in the nonrelativistic regime, i.e. $0 < \varepsilon \ll 1$. In this regime, the solution of the NKGE is highly oscillatory in time with $O(\varepsilon^2)$-wavelength, which brings significant difficulties in designing uniformly accurate numerical methods. The MTI-FP is based on (i) a multiscale decomposition by frequency of the NKGE in each time interval with simplified transmission conditions, and  (ii) an exponential wave integrator for temporal discretization and a Fourier pseudospectral method for spatial discretization. By adapting the energy method and the mathematical induction, we obtain two error bounds in $H^1$-norm at $O(h^{m_0} + \tau^2 / \varepsilon^2 )$ and $O(h^{m_0} + \tau + \varepsilon^2)$ with mesh size $h$, time step $\tau$ and $m_0$ an integer dependent on the regularity of the solution, which immediately implies a uniformly accurate error bound $O(h^{m_0} + \tau)$ with respect to $\varepsilon\in(0,1]$. 
In addition, by adopting a linear interpolation of the micro variables with the multiscale decomposition in each time interval, we obtain a uniformly accurate numerical solution for any time $t\ge0$.
Thus the proposed MTI-FP method has a super resolution property in time in terms of the Shannon sampling theory, i.e. accurate numerical solutions can be obtained even when the time step is much bigger than the $O(\varepsilon^2)$-wavelength. Extensive numerical results are reported to confirm our error bounds and demonstrate their super resolution in time. Finally the proposed MTI-FP method is applied to study numerically convergence rates of the NKGE to its different limiting models in the nonrelativistic regime.
\end{abstract}

\maketitle

\section{Introduction}
Consider the dimensionless nonlinear Klein-Gordon equation (NKGE) with the cubic nonlinearity 
in $d$-dimensions $(d = 1,2,3)$ \citep{li2025stability,scheider2020breather,
bao2019comparison,bao2012analysis,bainov1995nonexistence}
\begin{subequations}\label{1.1}
\begin{align}
&\varepsilon^2 \partial_{tt} u(\bx,t) - \Delta u(\bx,t) + \displaystyle \frac{1}{\varepsilon^2} u(\bx,t) + \lambda u(\bx,t)^3 = 0 , \quad \bx \in \mathbb{R}^d, \quad t > 0, \label{1.1a}\\
&u(\bx,0) = \phi_1(\bx), \quad \partial_t u(\bx,0) = \displaystyle \frac{1}{\varepsilon^2} \phi_2(\bx), \quad \bx \in \mathbb{R}^d, \label{1.1b}
\end{align}
\end{subequations}
where $t$ is the time, $\bx$ is the spatial variable, $u:=u(\bx,t)$ is a real-valued scalar field, $\lambda \in \mathbb{R}$ is a given constant, $0 < \varepsilon \le 1$ is a dimensionless parameter which is inversely proportional to the speed of light, and  $\phi_1(\bx)$ and $\phi_2(\bx)$ are two given real-valued initial data which are independent of $\varepsilon$.  

The NKGE \eqref{1.1} is known as a relativistic extension of the nonlinear Schr\"odinger equation \citep{masmoudi2002nonlinear,machihara2002nonrelativistic} and it has wide applications \citep{ohlsson2011relativistic,nakamura2001cauchy,brenner1981global}, which is used to describe the dynamics of spinless particles in quantum fields \citep{maireche2021solutions,pilkuhn2006klein} and serves as a model for dark-matter or black-hole evaporation problems in cosmology \citep{detweiler1980klein,ahmed2020generalized,bachelot2011klein}. It is time symmetric and conserves the energy \citep{bao2012analysis,bao2019comparison,machihara2002nonrelativistic,nakamura2001cauchy}, i.e. for $t\ge0$
\begin{align}
    E(t) & := \int_{\mathbb{R}^d} \left[  \varepsilon^2 | \partial_t u(\bx,t) |^2 + | \nabla u(\bx,t) |^2 + \displaystyle \frac{1}{\varepsilon^2} | u(\bx,t) |^2 +   \frac{\lambda}{2} | u(\bx,t) |^4 \right] d \bx \nonumber \\
    & \equiv \int_{\mathbb{R}^d} \left[  \displaystyle \frac{1}{\varepsilon^2} | \phi_2(\bx)|^2 + | \nabla \phi_1(\bx) |^2 + \displaystyle \frac{1}{\varepsilon^2} | \phi_1(\bx) |^2 +  \frac{\lambda}{2} | \phi_1(\bx) |^4 \right] d \bx 
    := E(0).
    \label{1.2}
\end{align}
For a fixed $\varepsilon = O(1) $, e.g. $\varepsilon = 1$, the NKGE has been well studied both analytically and numerically; see \citep{morawetz1972decay,jimenez1990analysis,duncan1997sympletic,li2025stability,nakamura2001cauchy,brenner1981global,strauss1978numerical,wang2022symmetric,bao2022uniform,cohen2008conservation} and references therein. 

On the contrary, when $0 < \varepsilon \ll 1$, i.e. in the nonrelativistic regime, due to that the energy $E(t) = O(\varepsilon^{-2})$ in \eqref{1.2} becomes unbounded as $\varepsilon \to 0$, the NKGE \eqref{1.1}
 is a singular limit problem when $\varepsilon \to 0$.  This nonrelativistic limit problem has been well studied asymptotically and analytically in the literatures \citep{tsutsumi1984nonrelativistic,najman1990nonrelativistic,machihara2002nonrelativistic,masmoudi2002nonlinear}.
In fact, when $0 < \varepsilon \ll 1$, the solution of the NKGE \eqref{1.1} can be decomposed as \citep{machihara2002nonrelativistic,masmoudi2002nonlinear}
\begin{equation}
u(\bx,t) = e^{it/\varepsilon^2} v(\bx,t) +e^{-it/\varepsilon^2}\overline{v(\bx,t)}  + r(\bx,t),\quad \bx \in \mathbb{R}^d, \quad t\ge0,
\label{1.3}
\end{equation}
with $i=\sqrt{-1}$ being the complex unit and 
$\bar{f}$ denoting the complex conjugate of $f$, where $v:=v(\bx,t)$ satisfies the nonlinear Sch\"{o}dinger equation (NLSE) or  the nonlinear Sch\"{o}dinger equation with wave operator (NLSW) as a limiting model, and the remainder $r:=r(\bx,t)=O(\varepsilon^2)$ \citep{tsutsumi1984nonrelativistic,najman1990nonrelativistic,machihara2002nonrelativistic,bao2014uniformly}. Based on these results, it is clear that, in the nonrelativistic regime, the solution of the NKGE \eqref{1.1} is highly oscillatory in time with $O(\varepsilon^2)$-wavelength, which brings significant challenges in designing and analyzing efficient and accurate numerical methods that
are uniformly accurate for $\varepsilon\in(0,1]$, especially in the nonrelativistic regime with $0<\varepsilon\ll1$. 



Over the last decade, great efforts have been devoted to analyze existing numerical methods and to propose new uniformly accurate numerical methods for the NKGE \eqref{1.1} in the nonrelativistic regime, with particular attention on how their error bounds depend explicitly on the mesh size $h$ and time step $\tau$ as well as the small parameter $\varepsilon$. 
Those error bounds can indicate spatial/temporal resolution (or meshing strategy) of different numerical methods for the NKGE \eqref{1.1} in the nonrelativistic regime.  In fact, different numerical methods have their advantages and disadvantage. The classical finite difference time domain (FDTD) methods \citep{bao2012analysis,duncan1997sympletic,bao2019comparison} perform very well when $\varepsilon =O(1)$, but
require meshing strategy $h = O(1)$ and $\tau = O(\varepsilon^3)$ \citep{bao2012analysis} in the nonrelativistic regime, and thus they are {\bf under-resolution} methods in time in terms of the Shannon sampling theory, i.e. they need $O(\varepsilon^{-1})\gg1$ grid points to capture per wave in time when $0<\varepsilon\ll1$. The limiting integrators (LIs) \citep{faou2014asymptotic} perform well when $0<\varepsilon \ll1$, but they fail to converge with $O(1)$ errors when $\varepsilon=O(1)$.  Both the exponential wave integrators (EWI) \citep{bao2012analysis,wang2022symmetric,bao2019comparison} and the time-splitting methods \citep{dong2014time,bao2019comparison} perform very well when $\varepsilon =O(1)$, and they
require meshing strategy $h = O(1)$ and $\tau = O(\varepsilon^2)$ in the nonrelativistic regime, 
and thus they are  {\bf optimal-resolution} methods in time in terms of the Shannon sampling theory, i.e. it needs $O(1)$ grid points to capture per wave in time when $0<\varepsilon\ll1$. In fact, all the above numerical methods 
are non-uniformly accurate numerical methods for the NKGE \eqref{1.1} with respect to $\varepsilon \in(0,1]$.
In recent years, several uniformly accurate numerical methods have been proposed and analyzed, including the multiscale time integrators (MTI) \citep{bao2014uniformly,bao2012uniformly}, the two-scale formulation (TSF) \citep{chartier2015uniformly} methods, multi-revolution composition (MRC) \citep{mauser2020rotating} methods and the uniformly accurate low regularity integrators (UALRI) \citep{calvo2022uniformly}. All of them share the same meshing strategy with $h=O(1)$ and $\tau=O(1)$ in the nonrelativistic regime, and thus they are {\bf super-resolution} methods in time in terms of the Shannon sampling theory, i.e. accurate solutions can be obtained even when the time step is much bigger than the temporal wavelength at $O(\varepsilon^2)$!

The main aim of this paper is to propose a new and simplified multiscale time integrator Fourier pseudospectral (MTI-FP) method for the NKGE \eqref{1.1} by adapting a multiscale decomposition by frequency of the NKGE in each time interval with simplified transmission conditions at each time step. We obtain two error bounds in $H^1$-norm at $O(h^{m_0} + \tau^2 / \varepsilon^2 )$ and $O(h^{m_0} + \tau + \varepsilon^2)$ with $m_0$ an integer dependent on the regularity of the solution by adapting two different analysis techniques, i.e. the energy method and the mathematical induction, respectively. From the two independent error bounds, we can immediately establish a uniformly accurate error bound at $O(h^{m_0} + \tau)$ with respect to $\varepsilon\in(0,1]$, which implies that the proposed MTI-FP method has super resolution property in time. Compared to the MTIs in the literature \citep{bao2014uniformly,bao2012uniformly}, the proposed MTI-FP has three main advantages: (i) the numerical scheme is significantly simplified and thus the computational cost is significantly reduced in practical computation,  (ii)
the regularity requirement is significantly weakened in order to obtain the same uniformly first order error bound in time, and (iii) it achieves optimal spatial accuracy with respect to the given regularity of the solution. 
In addition, by adopting a linear interpolation of the micro variables within the multiscale decomposition, 
i.e. $v$ and $r$ in \eqref{1.3}, in each time interval, we obtain a uniformly accurate approximation of the solution for any time $t\ge0$ with almost no additional computational cost. This post process for constructing numerical solution at any time $t\ge0$ with almost no additional computational cost is only applied to the MTI-FP methods, but not for the other uniformly accurate numerical methods, such as the TFS methods, MRC methods and the UALRI, in the literatures.



The paper is organized as follows. In Section~\ref{sec2}, we introduce the multiscale decomposition for the NKGE \eqref{1.1} with simplified transmission conditions. The corresponding MTI-FP method is proposed in Section~\ref{sec3}, and its rigorous error estimates are established in Section~\ref{sec4}. A multiscale interpolation is presented in Section \ref{sec5},  and numerical results are reported in Section~\ref{sec6}.
Finally, some conclusions are given in Section~\ref{sec7}. Throughout this paper, we adopt the standard Sobolev spaces and their corresponding norms, use c.c. to denote complex conjugate of the term in front of it,  and adopt $A \lesssim B$ to represent that there exists a generic constant $C > 0$ independent of $\varepsilon$, $h$ and $\tau$, such that $|A| \le CB$.


\section{Multiscale decompositions by amplitude and/or frequency} \label{sec2}  
In this section, we present different multiscale decompositions of the solution $u(\bx,t)$ of the NKGE 
\eqref{1.1} over $[0,T]$. 

Plugging the ansatz \eqref{1.3} into the NKGE \eqref{1.1}, after some simple computations, we have the following equation
\begin{align}
  0=&\varepsilon^2 \partial_{tt} u(\bx,t) - \Delta u(\bx, t) + \displaystyle \frac{1}{\varepsilon^2} u(\bx, t) + \lambda u(\bx, t)^3  \nonumber \\ 
  =&e^{it / \varepsilon^2} \left[ \varepsilon^2 \partial_{tt} v(\bx,t)+ 2i \partial_t v(\bx,t) - \Delta v(\bx,t) + 3 \lambda |v(\bx,t)|^2 v(\bx,t) \right]+{\rm c.c.} \nonumber \\ 
     &+\varepsilon^2 \partial_{tt} r(\bx,t) - \Delta r(\bx,t) + \frac{1}{\varepsilon^2} r(\bx,t) + \lambda 
     e^{3it / \varepsilon^2}v(\bx,t)^3 
     +\lambda e^{-3it / \varepsilon^2}\overline{v(\bx,t)}^3\nonumber\\
     &+ f(v(\bx,t),r(\bx,t),t),
     \label{2.1}
\end{align}
with
\begin{align}
f(v,r,t):=\lambda r\left[  r^2 + 3  r \left(e^{it / \varepsilon^2}v +e^{-it / \varepsilon^2} \overline{v} \right) + 3 \left( e^{it / \varepsilon^2}v + e^{-it / \varepsilon^2}\overline{v} \right)^2 \right],
\label{2.2}
\end{align}
and the following initial conditions
\begin{subequations}\label{init01}
\begin{align} 
&\phi_1(\bx) = v(\bx,0) + \overline{v(\bx,0)} + r(\bx,0), \quad \bx \in \mathbb{R}^d, \\[3pt]
&\displaystyle\frac{1}{\varepsilon^2}\phi_2(\bx) = \frac{i}{\varepsilon^2}\left( v(\bx,0) - \overline{v(\bx,0)} \right) + \partial_t v(\bx,0) + \overline{\partial_t v(\bx,0)} + \partial_t r(\bx,0).
\label{2.3b}
\end{align}
\end{subequations}
In order to make the remainder term $r$ to be small, e.g. $r=O(\varepsilon^2)$, we request 
\begin{equation}\label{rbx0}
r(\bx,0)\equiv 0  \quad \Rightarrow \quad v(\bx,0) + \overline{v(\bx,0)}=\phi_1(\bx) = O(1), \quad \bx \in \mathbb{R}^d.
\end{equation}
Due to the scale separation in \eqref{2.3b} when $0<\varepsilon\ll1$, we request
\begin{equation}
\label{ptrbx0}
v(\bx,0) - \overline{v(\bx,0)} =-i\phi_2(\bx), \quad \partial_t v(\bx,0) + \overline{\partial_t v(\bx,0)} + \partial_t r(\bx,0) = 0, \quad \bx \in \mathbb{R}^d.
\end{equation}
Combining \eqref{rbx0} and \eqref{ptrbx0}, we get
\begin{align}\label{initv23}
&v(\bx,0)=\frac{1}{2} \left[\phi_1 (\bx) - i \phi_2(\bx) \right]:=v_0(\bx), \quad \quad \bx \in \mathbb{R}^d,\\
&r(\bx,0)\equiv 0, \quad  \partial_t r(\bx,0)=-\partial_t v(\bx,0) - \partial_t \overline{v(\bx,0)}. 
\label{initr23}
\end{align}

By adopting a multiscale decomposition by amplitude with $O(1)$ and frequency with $\varepsilon^{-2}$ in \eqref{2.1}, we introduce two different decompositions and corresponding limiting models. 

Noting \eqref{initv23}-\eqref{initr23},
we request that $v$ satisfies the following nonlinear Schr\"{o}dinger equation (NLSE) with the initial condition \eqref{initv23}
\begin{align}\label{NLSE1}
2i \partial_t v(\bx,t) - \Delta v(\bx,t) + 3 \lambda |v(\bx,t)|^2 v(\bx,t)=0, \quad \quad \bx \in \mathbb{R}^d, \quad t > 0.
\end{align}
Combining \eqref{2.1} and \eqref{NLSE1}, noting \eqref{initv23}-\eqref{initr23} and \eqref{NLSE1} with $t=0$, we get that the remainder $r$ satisfies the nonlinear Klein-Gorden equation with source terms and small initial data:
\begin{subequations}\label{NKGEr1}
\begin{align}
&\varepsilon^2 \partial_{tt} r - \Delta r + \frac{1}{\varepsilon^2} r + f(v,r,t)+\left[\lambda 
     e^{3it / \varepsilon^2}v^3 
     +\varepsilon^2 e^{it / \varepsilon^2} \partial_{tt} v+{\rm c.c.}\right]=0,\\
&r(\bx,0)\equiv 0, \quad  \partial_t r(\bx,0)=\frac{1}{8}\left[4\Delta \phi_2(\bx)-3\lambda (\phi_1(\bx)^2+\phi_2(\bx)^2)\phi_2(\bx)\right].
\end{align}
\end{subequations}
It is easy to check that, if $v$ and $r$ are solutions of \eqref{NLSE1} with \eqref{initv23} and \eqref{NKGEr1}, respectively, then $u$ in \eqref{1.3} is a solution of the NKGE \eqref{1.1}. Thus the NLSE \eqref{NLSE1} with \eqref{initv23} is a limiting model of the NKGE \eqref{1.1} in the nonrelativistic regime. 
The convergence from the NKGE \eqref{1.1} to the NLSE \eqref{NLSE1} with \eqref{initv23} and the corresponding quadratic convergence rate with respect to $\varepsilon$ have been studied analytically and numerically
in the literatures \citep{bao2014uniformly,lei2023non,machihara2002nonrelativistic,masmoudi2002nonlinear}. We remark here that, due to the resonance term $e^{it / \varepsilon^2} \partial_{tt} v$ in 
\eqref{NKGEr1}, the amplitude of its solution $\|r(\cdot, t)\|_{L^\infty}$ may grow linearly or even quadratically in time \citep{lei2023non,machihara2002nonrelativistic,masmoudi2002nonlinear}, which might cause stability and efficiency issues in adopting the above multiscale decomposition with the NLSE \eqref{NLSE1} with \eqref{initv23} and the NKGE \eqref{NKGEr1} for designing uniformly accurate numerical methods for the original NKGE \eqref{1.1}.

To overcome the above difficulty, another multiscale decomposition by frequency for the NKGE \eqref{1.1} reads:
$v$ satisfies the following nonlinear Schr\"{o}dinger equation with wave operator (NLSW) as
\begin{subequations}\label{NLSW1}
\begin{align}
&\varepsilon^2 \partial_{tt} v(\bx,t)+2i \partial_t v(\bx,t) - \Delta v(\bx,t) + 3 \lambda |v(\bx,t)|^2 v(\bx,t)=0,\label{NLSW1eq} \\
&v(\bx,0)=v_0(\bx) = \frac{1}{2} \left[\phi_1 (\bx) - i \phi_2(\bx) \right], \quad \partial_t v(\bx,0)=\gamma(\bx),
\label{NLSW1eq22}
\end{align}
\end{subequations}
 and the remainder $r$ satisfies the nonlinear Klein-Gorden equation with source terms and small initial data:
\begin{subequations}\label{NKGEr2}
\begin{align}
&\varepsilon^2 \partial_{tt} r - \Delta r + \frac{1}{\varepsilon^2} r + f(v,r,t)+\left[\lambda 
     e^{3it / \varepsilon^2}v^3 +{\rm c.c.}\right]=0,\\
&r(\bx,0)\equiv 0, \quad  \partial_t r(\bx,0)=-\gamma(\bx)-\overline{\gamma(\bx)},
\end{align}
\end{subequations}
where $\gamma(\bx)$ is a given complex-valued function to be specified later. Again, it is easy to check that, if $v$ and $r$ are solutions of the NLSW \eqref{NLSW1} and NKGE \eqref{NKGEr2}, respectively, then $u$ in \eqref{1.3} is a solution of the NKGE \eqref{1.1}. Thus the NLSW \eqref{NLSW1} is a semi-limiting model of the NKGE \eqref{1.1} when $\varepsilon\to0$. 
The global-in-time quadratic convergence from the NKGE \eqref{1.1} to the NLSW \eqref{NLSW1} 
with respect to $\varepsilon$ have been studied analytically and numerically
in the literatures \citep{bao2014uniformly,lei2023non,machihara2002nonrelativistic,masmoudi2002nonlinear}. 

Different $\gamma(\bx)$ can be used in the multiscale decomposition by frequency \eqref{NLSW1} and \eqref{NKGEr2}.
In \citep{bao2014uniformly}, $\gamma(\bx)$ is taken as the well-prepared initial data of the NLSW, i.e. it is obtained by 
setting $t=0$ and $\varepsilon\to0$ in \eqref{NLSW1eq}, as\begin{align}
\gamma(\bx)&=\frac{i}{2}\left[-\Delta v(\bx,0) + 3 \lambda |v(\bx,0)|^2 v(\bx,0)\right]\nonumber\\
&=\frac{i}{2}\left[-\Delta v_0(\bx) + 3 \lambda |v_0(\bx)|^2 v_0(\bx)\right]:=v_1(\bx), \quad \bx \in \mathbb{R}^d.\label{2.12}
\end{align}
Based on this well-prepared initial data for the NLSW, formally one can show that
$v(\bx,t)$, $\partial_t v(\bx,t)$ and $\partial_{tt} v(\bx,t)$ are uniformly bounded with respect to $\varepsilon\in(0,1]$ \citep{bao2014uniformly}. Based on this choice of $\gamma(\bx)$ for the multiscale decomposition by frequency \eqref{NLSW1} and \eqref{NKGEr2}, a multiscale time integrator Fourier pseudospectral (MTI-FP) method was presented and analyzed in the literature \citep{bao2014uniformly}. On the other hand, we can take the most simplified initial condition as $\gamma(\bx)\equiv0$, which immediately implies $\partial_t v(\bx,0)=\partial_t r(\bx,0)\equiv 0$ in \eqref{NLSW1} and \eqref{NKGEr2}, i.e. the initial conditions for $v$ and $r$ are taken as homogeneous except $v(\bx,0)$.
Based on this most simplified initial condition, we will present a new and simplified MTI-FP method for the NKGE \eqref{1.1}  and 
establish its uniform first order accuracy in time with respect to $\varepsilon\in(0,1]$. Compared to the MTI-FP method in the literature \citep{bao2014uniformly}, 
due to the adopted homogeneous initial conditions, the proposed MTI-FP method is significantly simplified and thus the computational cost is greatly reduced in practical computation, and
the regularity requirement is significantly weakened in order to obtain the same uniformly first 
order error bound in time.

\section{A uniformly accurate MTI-FP method} \label{sec3}

 In this section, we present a multiscale time integrator Fourier pseudospectral (MTI-FP) method for the 
 NKGE \eqref{1.1} based on the multiscale decomposition by frequency \eqref{NLSW1} and \eqref{NKGEr2} 
 with $\gamma(\bx)\equiv0$. 
 
 For simplicity of notation and without loss of generality, we only present the MTI-FP method in one dimension (1D).
Generalization to high dimensions is straightforward by tensor product. As adopted in the literatures \citep{bao2014uniformly,faou2014asymptotic,bao2012analysis,duncan1997sympletic}, 
the NKGE \eqref{1.1} with $d = 1$ is usually truncated onto a bounded interval 	$\Omega=(a,b)$ ($|a|$ and $b$ are taken large enough such that the truncation error is negligible) with periodic boundary conditions
\begin{subequations}
\begin{align}
&\varepsilon^2 \partial_{tt} u(x,t) - \partial_{xx} u(x,t) + \displaystyle \frac{1}{\varepsilon^2} u(x,t) + \lambda u(x,t)^3 = 0 , \quad x \in \Omega, \quad t > 0, \\
&u(x,0) = \phi_1(x), \quad \partial_t u(x,0) = \displaystyle \frac{1}{\varepsilon^2} \phi_2(x), \quad x \in \overline{\Omega},\\
&u(a,t)=u(b,t), \qquad \partial_x u(a,t)=\partial_x u(b,t), \qquad t\ge0.
\end{align}
\label{NKGE1d}
\end{subequations}

\subsection{A multiscale decomposition via simplified transmission conditions}

Let $\tau=\Delta t>0$ be the time step size, and denote time levels by $t_n = n \tau$ for $n = 0,1,...$.
We first present a multiscale decomposition by frequency for the solution of \eqref{NKGE1d} on the time
interval $[t_n,t_{n+1}]$ with given initial data at $t=t_n$ as
\begin{equation}\label{inittn}
u(x,t_n)=\phi_1^n(x) =O(1), \qquad \partial_t u(x,t_n)=\frac{1}{\varepsilon^2}\phi_2^n(x)=O\left(\frac{1}{\varepsilon^2}\right), \qquad 
x\in \overline{\Omega}.
\end{equation}
Similar to those in the previous sections, we take an ansatz to the solution $u(x, t) := u(x, t_n + s)$ of
\eqref{NKGE1d} on the interval $[t_n,t_{n+1}]$ with \eqref{inittn} as
\begin{equation}
u(x,t_n+s) = e^{is/\varepsilon^2} v^n(x,s) +e^{-is/\varepsilon^2}\overline{v^n(x,s)}  + r^n(x,s),\quad x\in \overline{\Omega}, \quad 0\le s\le \tau.
\label{ansatz1D}
\end{equation}
Then a multiscale decomposition by frequency with simplified transmission conditions 
for the NKGE \eqref{NKGE1d} on the interval $[t_n,t_{n+1}]$ reads:
$v^n:=v^n(x,s)$ satisfies the following nonlinear Schr\"{o}dinger equation with wave operator (NLSW) with simplified transmission conditions as
\begin{subequations}\label{NLSWv1d}
\begin{align}
&\varepsilon^2 \partial_{ss} v^n(x,s) + 2i \partial_s v^n(x,s) - \partial_{xx} v^n(x,s) +3\lambda |v^n(x,s)|^2v^n(x,s)   =0, \\
&v^n (x,0) = \displaystyle\frac{1}{2} \left[\phi_1^n (x) - i \phi_2^n(x) \right], \quad  \partial_s v^n (x,0) = 0, \quad x \in \overline{\Omega},\label{NLSWv1d(b)}\\
&v^n(a,s) = v^n(b,s), \quad  \partial_x v^n(a,s) = \partial_x v^n(b,s), \quad 0 \le s \le \tau,
\end{align}
\end{subequations}
and the remainder $r^n:=r^n(x,s)$ satisfies the nonlinear Klein-Gorden equation 
with source terms and homogeneous initial data as
\begin{subequations}
\begin{align}
&\varepsilon^2 \partial_{ss}r^n - \partial_{xx} r^n + \displaystyle \frac{1}{\varepsilon^2} r^n +f(v^n,r^n,s)+
\left[\lambda e^{3is / \varepsilon^2} (v^n)^3 + {\rm c.c.} \right]=0, \\
&r^n(x,0) = 0,\quad \partial_s r^n(x,0) = 0, \quad x \in \overline{\Omega},\\
&r^n(a,s) = r^n(b,s), \quad  \partial_x r^n(a,s) = \partial_x r^n(b,s), \quad 0 \le s \le \tau,
\end{align}
\label{NKGEr1d}
\end{subequations}
where $f(v^n,r^n,s)$ is given in \eqref{2.2}.

\subsection{An exponential wave integrator Fourier pseudospectral full discretization}

Denote the mesh size as $h : = (b-a)/N$ with $N$ a positive even integer and grid points as $x_j := a + jh$ for $j = 0,1,...,N$. Define
\begin{align*}
   & X_N := \text{span} \left\{ e^{i\mu_l(x - a)} \ | \ l = \displaystyle -\frac{N}{2},..., \frac{N}{2} - 1 \right\} \quad \text{with} \ \ \mu_l : = \displaystyle \frac{2\pi l}{b-a}, \\
   & Y_N := \left\{ \mathbf{v} = (v_0,v_1,...,v_N)^T \in \mathbb{C}^{N+1} \ | \ v_0 = v_N  \right\} \  \text{with}  \ \left\|   \mathbf{v}
   \right\|_{l^2} : = \left( h \sum_{j = 0}^{N-1} | v_j|^2 \right)^{1/2}.
\end{align*}
For a periodic function $v(x)$ on $\overline{\Omega}$ and a vector $\mathbf{v} \in Y_N$, let $P_N:L^2(\Omega) \rightarrow X_N$ denote the standard $L^2 $-projection operator, and let $I_N: C(\Omega) \rightarrow X_N$ or $Y_N \rightarrow X_N$ be the trigonometric interpolation operator, i.e.,
\begin{align}
    (P_Nv)(x) = \sum_{l = -N / 2}^{N / 2 - 1} \widehat{v}_l e^{i \mu_l ( x - a)}, \ \ (I_N \mathbf{v})(x) = \sum_{l = -N / 2}^{N / 2 - 1} \widetilde{\mathbf{v}}_l e^{i \mu_l (x-a)}, \ \ a \le x \le b, 
\end{align}
where $\widehat{v}_l$ and $\widetilde{\mathbf{v}}_l$ are the Fourier and discrete Fourier transform coefficients of the periodic function $v(x)$ and the vector $\mathbf{v}$, respectively, defined as 
\begin{align}
    \widehat{v}_l = \displaystyle \frac{1}{b-a} \int_a^b v(x) e^{-i \mu_l (x - a)} dx, \quad \widetilde{\mathbf{v}}_l = \frac{1}{N} \sum_{j = 0}^{N-1} v_j e^{-i \mu_l (x_j - a)},
    \label{Fourier}
\end{align}
with $v_j:=v(x_j)$ for $j=0,1,\ldots,N$ for the periodic function $v(x)$.

To discretize the equations \eqref{NLSWv1d} and \eqref{NKGEr1d} numerically, we apply exponential wave integrators in time \citep{bao2012analysis} and Fourier spectral/pseudospectral method \citep{bernardi1997spectral,bao2014uniformly,shen2011spectral} in space.  
In the discretization, we first apply the Fourier spectral method for discretizing \eqref{NLSWv1d} and \eqref{NKGEr1d}. Our objective is to find $v^n_N(x,s) , \ r^n_N(x,s) \in X_N$ for $0 \le s \le \tau$, i.e.
\begin{align}
    v^n_N(x,s)  = \sum_{l = -N / 2}^{N / 2 - 1} \widehat{(v^n_N)}_l (s) e^{i \mu_l ( x - a)}, \quad r^n_N(x,s)  = \sum_{l = -N / 2}^{N / 2 - 1} \widehat{(r^n_N)}_l (s) e^{i \mu_l ( x - a)},
    \label{Fsol}
\end{align}
such that \eqref{Fsol} is the solution of the NLSW \eqref{NLSWv1d} and the remainder equation \eqref{NKGEr1d}. 

\noindent Expand the exact solutions $v^n$ and $r^n$ in Fourier series
\begin{align*}
   v^n(x,s)  = \sum_{l \in \mathbb{Z}} \widehat{(v^n)}_l (s) e^{i \mu_l ( x - a)}, \quad r^n(x,s)  = \sum_{l \in \mathbb{Z}} \widehat{(r^n)}_l (s) e^{i \mu_l ( x - a)},\qquad 0 \le s \le \tau.
\end{align*}
By taking the Fourier transform on both sides of \eqref{NLSWv1d} and \eqref{NKGEr1d}, and noticing the orthogonality of  $e^{i \mu_l (x -a)} $, we obtain for $l = -\frac{N}{2},...,\frac{N }{2} - 1$
\begin{eqnarray}
&&\quad \varepsilon^2 \widehat{(v^n)}_l''(s) + 2i \widehat{(v^n)}_l'(s) + \mu_l^2 \widehat{(v^n)}_l (s)+ \widehat{(g^n)}_l (s) = 0, \quad 0 < s \le \tau, \\
 &&\quad \varepsilon^2 \widehat{(r^n)}_l''(s) + \varepsilon^2 \omega_l^2\widehat{(r^n)}_l (s) + 
 e^{\frac{3is}{\varepsilon^2}}\widehat{(h^n)}_l(s) +e^{\frac{-3is}{\varepsilon^2}} \widehat{(\overline{h^n})}_l(s) 
  + \widehat{(f^n)}_l (s) = 0, 
\end{eqnarray}
where $\omega_l = \frac{1}{\varepsilon^2} \sqrt{1 + \varepsilon^2 \mu_l^2}$ for $l = -\frac{N}{2},...,\frac{N }{2} - 1$,
 and 
\begin{equation}\label{gnhnfn}
g^n:=g(v^n(x,s)), \quad h^n:= h(v^n(x,s)), \quad f^n:= f(v^n(x,s),r^n(x,s),s), 
\end{equation}
with 
\begin{equation}\label{gvnhvn}
    g(v^n):=3 \lambda |v^n|^2 v^n, \qquad h(v^n):= \lambda (v^n)^3. 
\end{equation}

Noting the initial data in \eqref{NLSWv1d} and \eqref{NKGEr1d}, we obtain by Duhamel's principle
\begin{eqnarray}\label{Sol1.v}
&&\widehat{(v^n)}_l(s) = a_l(s) \widehat{(v^n)}_l(0)  - \int_0^s b_l(s- \theta) \widehat{(g^n)}_l (\theta) d\theta, \quad 0 \le s \le \tau,  \\
&&\widehat{(r^n)}_l (s) =\int_0^s \frac{\mathrm{sin}(\omega_l (\theta- s))}{\varepsilon^2 \omega_l} \left[ 
e^{\frac{3i \theta}{\varepsilon^2}} \widehat{(h^n)}_l(\theta) +e^{\frac{-3i\theta}{\varepsilon^2}} \widehat{(\overline{h^n})}_l(\theta)  + \widehat{(f^n)}_l(\theta)  \right] d\theta, 
\label{Sol1.r}
\end{eqnarray}
where
\begin{align}
    \begin{cases}
        a_l(s) := \displaystyle \frac{\lambda_l^+ e^{is \lambda_l^-} - \lambda_l^-e^{is \lambda_l^+}}{\lambda_l^+ - \lambda_l^-}, \quad b_l(s) := i \frac{e^{is \lambda_l^+} - e^{is \lambda_l^-}}{\varepsilon^2 ( \lambda_l^- - \lambda_l ^+)}, \quad 0 \le s \le \tau,\\[2.2pt]
        \lambda_l^{\pm} := -  \displaystyle\frac{1}{\varepsilon^2}\left[
  1 \pm \sqrt{1 + \varepsilon^2 \mu_l^2}\right], \qquad l=-\frac{N}{2},\ldots, \frac{N}{2}-1.
    \end{cases}
    \label{Coeffab}
\end{align}
Then differentiating \eqref{Sol1.v}-\eqref{Sol1.r} with respect to $s$, we have
\begin{eqnarray}\label{DSol1.v}
&&\widehat{(v^n)}_l'(s) =  a_l'(s) \widehat{(v^n)}_l(0)  -  \displaystyle \int_0^s b_l'(s- \theta) \widehat{(g^n)}_l (\theta) d\theta, \quad 0 \le s \le \tau, \\
&&\widehat{(r^n)}_l' (s)  =  \int_0^s \frac{\mathrm{cos}(\omega_l (\theta- s))}{-\varepsilon^2} \left[ 
e^{\frac{3i \theta}{\varepsilon^2}} \widehat{(h^n)}_l (\theta) + e^{\frac{-3i \theta}{\varepsilon^2}} \widehat{(\overline{h^n})}_l (\theta) + \widehat{(f^n)}_l (\theta) \right] d \theta,
\label{DSol1.r}
\end{eqnarray}
where
\begin{align}
    a_l'(s) = i \lambda_l^+ \lambda_l^- \displaystyle \frac{e^{is \lambda_l^-} - e^{is \lambda_l^+}}{\lambda_l^+ - \lambda_l^-}, \quad b_l'(s) = \frac{\lambda_l^+ e^{is\lambda_l^+} - \lambda_l^- e^{is \lambda_l^-}}{\varepsilon^2 ( \lambda_l^+ - \lambda_l^- )}, \quad 0 \le s \le \tau.
    \label{Coeffdab}
\end{align}

To approximate the integrals in \eqref{Sol1.v}-\eqref{Sol1.r} and \eqref{DSol1.v}-\eqref{DSol1.r}, we apply (i) the Gautschi's type quadrature \citep{hairer2006geometric,bao2012uniformly,gautschi1961numerical} for integrals involving $g^n$ and $h^n$, and (ii) the trapezoidal rule \citep{bao2012uniformly,gautschi1961numerical} for integrals involving $f^n$, e.g.,
\begin{align*}
& \int_0^sb_l(s - \theta)\widehat{(g^n)}_l(\theta)d\theta \approx \int_0^s b_l(s - \theta) \left[ \widehat{(g^n)}_l(0) + \theta \widehat{(g^n)}_l'(0)  \right] d\theta, \\
& \int_0^s\frac{\mathrm{sin}(\omega_l(\theta-s))}{\varepsilon^2 \omega_l} \widehat{(f^n)}_l(\theta) d\theta \approx - s\frac{\mathrm{sin}(\omega_l s)}{2\varepsilon^2 \omega_l} \widehat{(f^n)}_l(0).
\end{align*}
Thus by taking $s = \tau$, we obtain
\begin{equation}
\begin{aligned} 
& \widehat{(v^n)}_l(\tau) \approx a_l(\tau) \widehat{(v^n)}_l(0) - c_l(\tau) \widehat{(g^n)}_l (0), \\
& \widehat{(v^n)}_l'(\tau) \approx a_l'(\tau) \widehat{(v^n)}_l(0)  - c_l'(\tau) \widehat{(g^n)}_l (0), 
\end{aligned}\label{Solv1}  
\end{equation}
and
\begin{subequations}\label{Solr1}
\begin{align}
&\widehat{(r^n)}_l (\tau) \approx \displaystyle  - p_l(\tau)  \widehat{(h^n)}_l (0)  - \overline{p_l}(\tau)  \widehat{(\overline{h^n})}_l (0),  \\
&\widehat{(r^n)}_l' (\tau) \approx \displaystyle   - p_l'(\tau)  \widehat{(h^n)}_l (0)   - \overline{p_l'}(\tau)  \widehat{(\overline{h^n})}_l (0)  - \displaystyle \frac{\tau}{2\varepsilon^2} \widehat{(f^n)}_l(\tau),
\end{align}
\end{subequations}
where we apply the homogeneous transmission condition $\partial_s v^n(x,0) = 0$ and thus $\left.\partial_s g^n\right|_{s=0} =\left. \partial_s h^n\right|_{s=0} = 0$, and
\begin{align}
\begin{cases}
c_l(\tau) = \displaystyle \int_0^\tau b_l (\tau - \theta) d \theta, \quad &  p_l(\tau) = \displaystyle \int_0^{\tau}  \frac{\text{sin}(\omega_l (\tau - \theta))}{\varepsilon^2 \omega_l} e^{3i\theta / \varepsilon^2} d \theta,   \\[3pt]
c_l'(\tau) = \displaystyle \int_0^\tau b_l'(\tau - \theta) d\theta, \quad & p_l'(\tau) = \displaystyle \int_0^{\tau}  \frac{\text{cos}(\omega_l (\tau - \theta))}{\varepsilon^2 } e^{3i\theta / \varepsilon^2} d \theta.
 \end{cases}
 \label{Coeffcp}
\end{align}
Combining \eqref{Fsol}, \eqref{Solv1}-\eqref{Solr1} and the multiscale decomposition \eqref{ansatz1D}, we obtain the multiscale time integrator Fourier spectral (MTI-FS) method for the NKGE \eqref{NKGE1d}, i.e.,
\begin{subequations}
\begin{align}
\widehat{(v^n_N)}_l(\tau) & \approx a_l(\tau) \widehat{(v^n_N)}_l(0) - c_l(\tau) \widehat{(g^n_N)}_l (0),\\ 
\widehat{(v^n_N)}_l'(\tau) & \approx a_l'(\tau) \widehat{(v^n_N)}_l(0)  - c_l'(\tau) \widehat{(g^n_N)}_l (0),\\
\widehat{(r^n_N)}_l (\tau) & \approx \displaystyle  - p_l(\tau)  \widehat{(h^n_N)}_l (0)  - \overline{p_l}(\tau)  \widehat{(\overline{h^n_N})}_l (0), \qquad l = - \frac{N}{2}, ...,\frac{N}{2}-1, \\
\widehat{(r^n_N)}_l' (\tau) & \approx \displaystyle   - p_l'(\tau)  \widehat{(h^n_N)}_l (0)   - \overline{p_l'}(\tau)  \widehat{(\overline{h^n_N})}_l (0)  - \displaystyle \frac{\tau}{2\varepsilon^2} \widehat{(f^n_N)}_l(\tau),
\end{align}
\end{subequations}
with $g^n_N:=g(v^n_N(x,s))$, $h^n_N:= h(v^n_N(x,s))$ and $f^n_N:= f(v^n_N(x,s),r^n_N(x,s),s)$.

In practice, the integrals in \eqref{Fourier} are usually approximated by numerical quadratures \citep{bernardi1997spectral,shen2011spectral,gautschi1961numerical}. Let $u_j^n$ and $\dot{u}_j^n$ be approximations of $u(x_j,t_n)$ and $\partial_t u(x_j,t_n)$, respectively, and let $v^{n,1}_j$, $\dot{v}^{n,1}_j$, $r^{n,1}_j$,  $\dot{r}^{n,1}_j$ be approximations of $v^n(x_j,\tau)$, $\partial_s v^n(x_j,\tau)$, $r^n(x_j,\tau)$ and $\partial_s r^n(x_j,\tau)$, respectively, for $n = 0,1,..$ and $j = 0,1,...,N$. Denote $\mathbf{u}^n=(u_0^n,u_1^n,\ldots,u_N^n)^T\in Y_N$, $\dot{\mathbf{u}}^n=(\dot{u}_0^n,\dot{u}_1^n,\ldots,
\dot{u}_N^n)^T\in Y_N$, $\mathbf{v}^{n,1}=(v_0^{n,1},v_1^{n,1},\ldots$ $,v_N^{n,1})^T\in Y_N$, $\dot{\mathbf{v}}^{n,1}=(\dot{v}_0^{n,1},\dot{v}_1^{n,1},\ldots,\dot{v}_N^{n,1})^T\in Y_N$, $\mathbf{r}^{n,1}=(r_0^{n,1},r_1^{n,1},\ldots,r_N^{n,1})^T\in Y_N$ and $\dot{\mathbf{r}}^{n,1}=(\dot{r}_0^{n,1},$ $\dot{r}_1^{n,1},\ldots,\dot{r}_N^{n,1})^T\in Y_N$.
Then a multiscale time integrator Fourier pseudospectral (MTI-FP) method for the NKGE \eqref{NKGE1d} reads as
\begin{subequations}\label{MTIu}
    \begin{align}
&u_j^{n+1} = e^{i\tau / \varepsilon^2} v^{n,1}_j + e^{-i\tau / \varepsilon^2} \overline{v^{n,1}_j} + r_j^{n,1}, \quad j = 0,1,...,N, \\
&\dot{u}_j^{n+1} = e^{i\tau / \varepsilon^2} \left( \dot{v}_j^{n,1} + \displaystyle \frac{i}{\varepsilon^2} v^{n,1}_j \right) + e^{-i\tau / \varepsilon^2} \left( \overline{\dot{v}_j^{n,1}} - \displaystyle \frac{i}{\varepsilon^2} \overline{v^{n,1}_j} \right) + \dot{r}_j^{n,1},
    \end{align}
\end{subequations}
where
\begin{subequations}\label{MTIvr1}
\begin{align}
& v_j^{n,1} = \sum_{l = -N / 2}^{N/2-1} \widetilde{(\mathbf{v}^{n,1})}_l e^{i \mu_l (x_j - a)}, \quad & \dot{v}_j^{n,1} = \sum_{l = -N / 2}^{N/2 - 1} \widetilde{(\dot{\mathbf{v}}^{n,1})}_l e^{i \mu_l (x_j -a)}, \\[3pt]
&   r_j^{n,1} = \sum_{l = -N / 2}^{N/2-1} \widetilde{(\mathbf{r}^{n,1})}_l e^{i \mu_l (x_j - a)}, \quad & \dot{r}_j^{n,1} = \sum_{l = -N / 2}^{N/2 - 1} \widetilde{(\dot{\mathbf{r}}^{n,1})}_l e^{i \mu_l (x_j -a)},
    \end{align}
\end{subequations}
with the Fourier coefficients given by
\begin{subequations}\label{MTIvr2}
\begin{align}
 \widetilde{(\mathbf{v}^{n,1})}_l & = a_l(\tau) \widetilde{(\mathbf{v}^{n,0})}_l  - c_l(\tau) \widetilde{(\mathbf{g}^{n,0})}_l, \  
\widetilde{(\dot{\mathbf{v}}^{n,1})}_l = a_l'(\tau) \widetilde{(\mathbf{v}^{n,0})}_l  - c_l'(\tau) \widetilde{(\mathbf{g}^{n,0})}_l , \label{MTIvr2.v}\\
 \widetilde{(\mathbf{r}^{n,1})}_l & = \displaystyle  - p_l(\tau)  \widetilde{(\mathbf{h}^{n,0})}_l   - \overline{p_l}(\tau)  \widetilde{(\overline{\mathbf{h}^{n,0}})}_l , \qquad l = \displaystyle -\frac{N}{2},..., \frac{N}{2}-1,  \\
\widetilde{(\dot{\mathbf{r}}^{n,1})}_l &  = \displaystyle   - p_l'(\tau)  \widetilde{(\mathbf{h}^{n,0})}_l   - \overline{p_l'}(\tau)  \widetilde{(\overline{\mathbf{h}^{n,0}})}_l  
- \displaystyle \frac{\tau}{2\varepsilon^2} \widetilde{(\mathbf{f}^{n,1})}_l,
\end{align}
\end{subequations}
and
\begin{subequations}\label{MTIIni.1}
\begin{align}
&\mathbf{v}^{n,0}=(v^{n,0}_0,v^{n,0}_1,\ldots,v^{n,0}_N)^T\in Y_N, \quad 
\mathbf{g}^{n,0}=(g^{n,0}_0,g^{n,0}_1,\ldots,g^{n,0}_N)^T\in Y_N, \\ 
&\mathbf{h}^{n,0}=(h^{n,0}_0,h^{n,0}_1,\ldots,h^{n,0}_N)^T\in Y_N, \quad 
\mathbf{f}^{n,1}=(f^{n,1}_0,f^{n,1}_1,\ldots,f^{n,1}_N)^T\in Y_N,
\end{align}
\end{subequations}
where
\begin{subequations}\label{MTIIni.2}
\begin{align}
& v^{n,0}_j = \displaystyle \frac{1}{2} \left(u^n_j - i\varepsilon^2 \dot{u}^n_j  \right), \qquad j = 0,1,...,N,\qquad n\ge0,\label{MTIIni1} \\
&g^{n,0}_j = g(v^{n,0}_j), \quad h^{n,0}_j = h(v^{n,0}_j), \quad f^{n,1}_j = f(v^{n,1}_j,r^{n,1}_j,\tau). \label{MTIIni2}
\end{align}
\end{subequations}
The initial data is given as
\begin{equation}\label{init00}
u^0_j=\phi_1(x_j), \qquad \dot{u}^0_j=\frac{1}{\varepsilon^2}\phi_2(x_j), \qquad j=0,1,\ldots, N.
\end{equation}

In practical computation, the above MTI-FP method for \eqref{NKGE1d} is implemented by the fast Fourier transform (FFT). Thus, this method is fully explicit, accurate and very efficient, and its memory cost is $O(N)$ and the computational cost per time step is $O(N  \text{log} N)$.

\begin{remark}
Compared to the MTI-FP method with well-prepared initial data \citep{bao2014uniformly}, the choice of homogeneous transmission conditions $\gamma(x) \equiv 0$ greatly simplifies the scheme and the corresponding error analysis. For practical computation, it improves computational efficiency and reduces runtime.
\end{remark}


\section{A uniformly accurate error bound} \label{sec4}
In this section, we establish two independent optimal error bounds for the MTI-FP \eqref{MTIu}-\eqref{init00} by two different techniques, which immediately imply a uniformly accurate error bound 
with respect to $\varepsilon\in(0,1]$.

\subsection{The main result}

Let $0 < T < T^*$ with $T^*$ the maximum existence time of the solution $u(x,t)$ to the problem \eqref{NKGE1d}. We make the following assumption (A) of the solution $u(x,t)$, i.e. there exists an integer $m_0 \ge 4$ such that
\begin{align}
 \text{(A)} \quad \ u \in C^1 \left([0,T]; H^{m_0+1}_p (\Omega) \right), \quad \left\| u \right\|_{L^{\infty} \left( [0,T]; H^{m_0 +1} \right)} + \varepsilon^2 \left\| \partial_t u \right\|_{L^{\infty} \left( [0,T]; H^{m_0 +1} \right)} \lesssim 1, \nonumber
\end{align}
where $H_p^m(\Omega) = \left\{ \phi(x) \in H^m(\Omega) \ | \ \phi^{(k)}(a) = \phi^{(k)} (b), \ k = 0,1,...,m-1  \right\}$.

We introduce the following notation
\begin{align}
C_0 = \max_{0<\varepsilon \le 1} \left\{ \left\| u \right\|_{L^{\infty} \left( [0,T]; H^{m_0 +1} \right)} ,  \varepsilon^2 \left\| \partial_t u \right\|_{L^{\infty} \left( [0,T]; H^{m_0 +1} \right)}\right\}, \label{4.2}
\end{align}
and define the error functions as
\begin{align}
e^n(x) := u(x,t_n) - ( I_N \mathbf{u}^n ) (x), \quad \dot{e}^n(x) := \partial_t u(x,t_n) -( I_N \dot{\mathbf{u}}^n ) (x), \quad x \in \overline{\Omega},  \label{errfun}
\end{align}
where $\mathbf{u}^n$ and  $\dot{\mathbf{u}}^n$
are the numerical solutions obtained from the MTI-FP method \eqref{MTIu}-\eqref{init00}.
Then we have the following error bounds for the MTI-FP \eqref{MTIu}-\eqref{init00}.

\begin{theorem} Under the assumption (A), there exist two constants $0 < h_0 < 1$ and $0 < \tau_0 < 1$ sufficiently small and independent of $\varepsilon$ such that for any $0 < \varepsilon \le 1$, when $0 < h \le h_0$ and $0 < \tau \le \tau_0$, we have
\begin{align}
    & \left\| e^n \right\|_{H^1} + \varepsilon^2 \left\| \dot{e}^n \right\|_{H^1} \lesssim h^{m_0 } + \displaystyle \frac{\tau^2}{\varepsilon^2}, \quad  \left\| e^n \right\|_{H^1} + \varepsilon^2 \left\| \dot{e}^n \right\|_{H^1} \lesssim h^{m_0 } + \tau + \varepsilon^2, \label{thm1err1} \\
    & \left\| I_N \mathbf{u}^n \right\|_{H^1} \le C_0 + 1, \quad \left\| I_N \dot{\mathbf{u}}^n \right\|_{H^1} \le \displaystyle \frac{1}{ \varepsilon^2}(C_0 + 1) ,\quad 0 \le n \le \displaystyle \frac{T}{\tau}. \label{thm1err2}
\end{align}
Thus, by taking the minimum of the two error estimates in \eqref{thm1err1} and then taking the maximum for $\varepsilon \in (0,1]$, we obtain an error estimate uniformly for $\varepsilon \in (0,1]$,
\begin{align}
    \left\| e^n \right\|_{H^1} + \varepsilon^2 \left\| \dot{e}^n \right\|_{H^1} & \lesssim h^{m_0 } + \tau + \max_{0 < \varepsilon \le 1}\min \left\{ \displaystyle \frac{\tau^2}{\varepsilon^2}, \varepsilon^2 \right\} \nonumber \\
    & \lesssim h^{m_0} + \tau, \qquad 0 \le n \le \displaystyle \frac{T}{\tau}.\label{thm1err3}
\end{align}
\label{tm4.1}
\end{theorem}

In order to prove Theorem \ref{tm4.1}, we introduce the error energy functional \citep{bao2014uniformly}
\begin{align}
    \mathcal{E}( e^n,\dot{e}^n) := \varepsilon^2 \left\| \dot{e}^n \right\|_{H^1}^2 + \left\|\partial_x e^n \right\|_{H^1}^2 + \displaystyle \frac{1}{\varepsilon^2} \left\| e^n \right\|_{H^1}^2, \quad 0 \le n \le \frac{T}{\tau}. \label{enefun}
\end{align}
The proof will be split into four main steps to be presented in the following subsections as:
(i) estimates for the micro (or local) variables $v^n$ and $r^n$, (ii)
 error functions and the corresponding error equations, (iii) estimates for local truncation errors and nonlinear terms errors, and 
(iv) proof for the main result by the energy method.

\subsection{Estimates for the micro (or local) variables }

We first show prior estimates for $v^n$ and $r^n$ in the multiscale decomposition \eqref{NLSWv1d}-\eqref{NKGEr1d} at each time step. 

\begin{lemma}\label{lemma4.1}
(Prior estimates) Under the assumption (A), there exists a constant $\tau_1 > 0$ independent of $0 < \varepsilon \le 1$ and $h > 0$, such that for $0 < \tau \le \tau_1$
\begin{align}
    &\left\| v^n \right\|_{L^{\infty}([0,\tau]; H^{m_0 + 1})} + \left\| \partial_s v^n \right\|_{L^{\infty}([0,\tau]; H^{m_0 - 1})} + \varepsilon^2 \left\| \partial_{ss} v^n \right\|_{L^{\infty}([0,\tau]; H^{m_0-1})} \lesssim 1, \label{lm1v} \\ 
    & \left\| r^n \right\|_{L^{\infty}([0,\tau]; H^{m_0 - 1})} + \varepsilon^2 \left\| \partial_s r^n \right\|_{L^{\infty}([0,\tau]; H^{m_0 - 1})} + \varepsilon^4 \left\| \partial_{ss} r^n \right\|_{L^{\infty}([0,\tau]; H^{m_0 -3})} \lesssim \varepsilon^2. \label{lm1r}
\end{align}
\end{lemma}
\begin{proof}
Recalling the initial data \eqref{NLSWv1d(b)} and the assumption (A), we have
\begin{align*}
\left\| v^n(\cdot,0) \right\|_{H^{m_0 + 1}} \lesssim \left\| u(\cdot,t_n) \right\|_{H^{m_0 + 1}}+\varepsilon^2 \left\| \partial_t u(\cdot,t_n) \right\|_{H^{m_0 + 1}} \lesssim 1.
\end{align*}
Applying Duhamel's principle to the NLSW \eqref{NLSWv1d}, we get
\begin{align*}
    v^n(x,s) = a(s) v^n(x,0)  - \int_0^s b(s-\theta) g(v^n(x,\theta)) d\theta, \quad0 \le s \le \tau,
\end{align*}
where $a(s)$ and $b(s)$ are the pseudo-differential operators
\begin{align*}
    a(s) = \displaystyle \frac{\lambda^+ e^{is \lambda^-} -\lambda^- e^{is \lambda^+}}{\lambda^+ - \lambda^-}, \quad b(s) = i \frac{e^{is \lambda^+} - e^{is \lambda^-}}{\varepsilon^2(\lambda^- - \lambda^+)}, \quad \lambda^\pm = - \frac{1 \pm\sqrt{1 - \varepsilon^2 \Delta}}{\varepsilon^2},
\end{align*}
and for any $\phi \in H^k (\Omega)$
\begin{align*}
    \left\| a(s) \phi \right\|_{H^k} \lesssim \left\| \phi \right\|_{H^k},\quad \left\| b(s) \phi \right\|_{H^k} \lesssim \left\| \phi \right\|_{H^k}.
\end{align*}
Then we have
\begin{align*}
    \left\| v^n(\cdot,s) \right\|_{H^{m_0 + 1}} \lesssim \left\| v^n(\cdot,0) \right\|_{H^{m_0 + 1}}  + \int_0^s \left\| g(v^n(\cdot,\theta
    )) \right\|_{H^{m_0 + 1}} d \theta, \quad 0 \le s \le \tau.
\end{align*}
Thus by a standard bootstrap argument for the nonlinear wave equation \citep{tao2006local}, we obtain that there exists a positive constant $\tau^*_1$ independent of $\varepsilon$ and $h$ such that for any $0 < \tau \le \tau_1^*$,
\begin{align*}
    \left\| v^n \right\|_{L^{\infty}([0,\tau]; H^{m_0 + 1})} \lesssim 1.
\end{align*}
Similarly, by differentiating \eqref{NLSWv1d}, we can get the formula for $\partial_s v^n$, i.e.,
\begin{align*}
    \partial_sv^n(x,s) = \varepsilon^2 b(s) \partial_{ss} v^n(x,0) - \int_0^s b(s-\theta) \partial_{\theta}g(v^n(x,\theta)) d\theta,\quad 0 \le s \le \tau,
\end{align*}
where
\begin{align*}
    \partial_{ss} v^n(x,0) = \frac{1}{\varepsilon^2} \left[ \partial_{xx} v^n(x,0) - g(v^n(x,0)) \right] \in H^{m_0 -1}.
\end{align*}
Thus with 
\begin{align*}
    \varepsilon^2 \partial_{ss} v^n(x,s) =- 2i \partial_s v^n(x,s) + \partial_{xx} v^n(x,s) - 3\lambda |v^n(x,s)|^2v^n(x,s), \quad 0 \le s \le \tau,
\end{align*}
we can easily establish estimates \eqref{lm1v} in a similar manner with details omitted here for brevity.

To obtain the estimate \eqref{lm1r}, we perform the analysis in Fourier space. Assuming 
\begin{align*}
    r^n(x,s) = \sum_{l \in \mathbb{Z}} \widehat{(r^n)}_l (s) e^{i \mu_l (x-a)}, \quad x \in \overline{\Omega}, \quad 0 \le s \le \tau,
    \end{align*}
and taking the Fourier transform on both sides of \eqref{NKGEr1d}, we have
\begin{align}
\widehat{(r^n)}_l (s)  = \int_0^s \frac{\mathrm{sin}(\omega_l ( \theta-s ))}{\varepsilon^2 \omega_l} 
\left( e^{\frac{3i \theta}{\varepsilon^2}} \widehat{(h^n)}_l (\theta) + 
e^{\frac{-3i \theta}{\varepsilon^2}} \widehat{( \overline{h^n})}_l (\theta) 
+ \widehat{(f^n)}_l (\theta) \right) d \theta,  \label{lm1.1}
\end{align}
where $h^n$ and $f^n$ are given in \eqref{gnhnfn}. 

Let $\beta_l = \omega_l - \frac{1}{\varepsilon^2} = \mu_l^2/\sqrt{1 + \varepsilon^2 \mu_l^2}$, then we get
\begin{align*}
    &\int_0^s \frac{\mathrm{sin}(\omega_l ( \theta-s))}{\varepsilon^2 \omega_l} e^{3i \theta / \varepsilon^2} \widehat{(h^n)}_l (\theta) d\theta \nonumber \\
    &=i\int_0^s \frac{e^{i\omega_l(s - \theta)}-e^{-i\omega_l(s - \theta)}}{2\varepsilon^2 \omega_l} e^{3i \theta / \varepsilon^2} \widehat{(h^n)}_l (\theta) d\theta \nonumber \\
    &=\displaystyle \frac{i}{2\varepsilon^2 \omega_l} \left[e^{i \omega_l s}  \int_0^s e^{-i \beta_l \theta} e^{2i \theta / \varepsilon^2} \widehat{(h^n)}_l (\theta) d\theta-e^{-i \omega_l s} \int_0^s e^{i \beta_l \theta} e^{4i \theta / \varepsilon^2} \widehat{(h^n)}_l (\theta) d\theta\right] \nonumber \\
   & =\displaystyle \frac{i}{2\varepsilon^2 \omega_l} \left[ \frac{\varepsilon^2}{2i} e^{i \omega_l s}  \int_0^s  e^{-i \beta_l \theta} \widehat{(h^n)}_l (\theta)  d(e^{2i \theta / \varepsilon^2}) - \frac{\varepsilon^2}{4i} e^{-i \omega_l s} \int_0^s e^{i \beta_l \theta} \widehat{(h^n)}_l (\theta) d(e^{4i \theta / \varepsilon^2}) \right] \nonumber \\
   &=\displaystyle \frac{ e^{i \omega_l s}}{4 \omega_l} \left[ e^{is(\frac{2}{\varepsilon^2}- \beta_l)} \widehat{(h^n)}_l (s)-\widehat{(h^n)}_l(0)- \int_0^s e^{i\theta(\frac{2}{\varepsilon^2}-\beta_l)} \left( \widehat{(\partial_{\theta} h^n)}_l -i\beta_l \widehat{(h^n)}_l \right)  d \theta \right]
   \nonumber \\
    &\ \ +\frac{e^{-i \omega_l s}}{8 \omega_l} \left[ \widehat{(h^n)}_l(0)- e^{is(\beta_l+\frac{4}{\varepsilon^2})} \widehat{(h^n)}_l (s) +\int_0^s e^{i\theta(\beta_l+ \frac{4}{\varepsilon^2})} \left( \widehat{(\partial_{\theta} h^n)}_l + i\beta_l \widehat{(h^n)}_l \right) 
    d \theta \right]. 
\end{align*}
Thus we have
\begin{align*}
& \left| \int_0^s \frac{\mathrm{sin}(\omega_l (\theta- s ))}{\varepsilon^2 \omega_l} e^{3i \theta / \varepsilon^2} \widehat{(h^n)}_l (\theta) d\theta \right| \nonumber \\
& \lesssim  \varepsilon^2 \left[ \left | \widehat{(h^n)}_l(s) \right| + \left| \widehat{(h^n)}_l(0) \right| + \int_0^s \left(\left| \widehat{(\partial_{\theta} h^n )}_l (\theta) \right| + \mu_l^2 \left| \widehat{(h^n)}_l (\theta) \right|\right) d\theta \right].
\end{align*}
Similarly, by integration by parts, we obtain the following estimates for \eqref{lm1.1}
\begin{align}
\left| \widehat{(r^n)}_l(s) \right|\lesssim & \varepsilon^2 \left[ \left | \widehat{(h^n)}_l(s) \right| + \left| \widehat{(h^n)}_l(0) \right| + \int_0^s \left(\left| \widehat{(\partial_{\theta} h^n )}_l (\theta) \right| + \mu_l^2 \left| \widehat{(h^n)}_l (\theta) \right| \right)d\theta \right] \nonumber \\
& + \varepsilon^2 \left[ \left| \widehat{(\overline{h^n})}_l(s) \right| + \left| \widehat{(\overline{h^n})}_l(0) \right| + \int_0^s \left| \widehat{(\overline{\partial_{\theta} h^n })}_l (\theta) \right| + \mu_l^2 \left| \widehat{(\overline{h^n})}_l (\theta) \right| d\theta \right] \nonumber \\
& + \displaystyle \int_0^s \left| \widehat{(f^n)}_l (\theta) \right| d \theta, \quad 0\le s\le \tau,
    \quad l \in \mathbb{Z}.\label{lm1.2}
\end{align}
Multiplying the square of \eqref{lm1.2} by $1 + \mu_l^2 + \cdot \cdot \cdot + \mu_l^{2( m_0 -1)}$, and summing them up for $l \in \mathbb{Z}$, we get
\begin{align}
\left\| r^n (\cdot,s) \right\|_{H^{m_0 - 1}} ^2\lesssim &\varepsilon^4 \left[     \left\| h^n \right\|_{L^{\infty}([0,\tau];H^{m_0 -1})}^2+  s^2 \left\| \partial_s h^n \right\|_{L^{\infty}([0,\tau];H^{m_0 - 1})}^2 \right. \nonumber \\
& \left. + s^2 \left\| h^n \right\|_{L^{\infty}([0,\tau];H^{m_0 + 1})}^2  \right] + s\int_0^s \left\| f^n (\theta)\right\|_{H^{m_0 - 1}}^2 d \theta  \nonumber \\
\lesssim & \varepsilon^4+\int_0^s \left\| f^n (\theta) \right\|_{H^{m_0 - 1}}^2 d \theta, \qquad 0 \le s \le \tau.
\label{lm1.3}
\end{align}
Recalling the formula of $f^n$ in \eqref{gnhnfn} and the prior estimate \eqref{lm1v}, and adapting the bootstrap argument \citep{tao2006local}, we obtain that there exists a positive constant $0 < \tau_1 <\tau_1^* $ independent of $\varepsilon$ and $h$ such that for any $0 < \tau \le \tau_1$,
\begin{align*}
    \left\| r^n \right\|_{L^{\infty}([0,\tau];H^{m_0 - 1})} \lesssim \varepsilon^2.
\end{align*}

\noindent Similarly we can obtain the following estimates
\begin{align}
    \left\| \partial_s r^n \right\|_{L^{\infty}([0,\tau];H^{m_0 - 1})} \lesssim 1, \quad \left\| \partial_{ss} r^n \right\|_{L^{\infty}([0,\tau];H^{m_0 - 3})} \lesssim \displaystyle \frac{1}{\varepsilon^2},
    \label{4.15}
\end{align} 
which completes the proof of \eqref{lm1r}. 
\end{proof}

\subsection{Error functions and the corresponding error equations}

Define another set of error functions
\begin{equation}
e_N^n := (P_N u)(x,t_n) - ( I_N \mathbf{u}^n ) (x), \ 
\dot{e}_N^n := (P_N \partial_t u)(x,t_n) - ( I_N \mathbf{\dot{u}}^n ) (x), \ x \in \overline{\Omega}.
\end{equation} 
By the regularity of the solution $u(x,t)$ in Assumption (A), we have
\begin{subequations}\label{Tri_err}
\begin{align}
    &\left\| e^n \right\|_{H^1} \le \left\| u(\cdot,t_n) - (P_N u)(\cdot,t_n) \right\|_{H^1} + \left\| e^n_N \right\|_{H^1} \lesssim h^{m_0} + \left\| e_N^n \right\|_{H^1} , \\ 
   & \left\| \dot{e}^n \right\|_{H^1} \le \left\| \partial_t u(\cdot,t_n) - (P_N \partial_t u)(\cdot,t_n) \right\|_{H^1} + \left\| \dot{e}_N^n \right\|_{H^1} \lesssim \displaystyle \frac{h^{m_0}}{\varepsilon^2} + \left\| \dot{e}^n_N \right\|_{H^1}.
\end{align}
\end{subequations}
Thus, we only need to prove estimates \eqref{thm1err1} and \eqref{thm1err2} with $e^n$ and $\dot{e}^n$
replaced by $e^n_N$ and $\dot{e}_N^n$, respectively.

\begin{lemma}\label{lemma4.2}
(Formula of the exact solution) Denote the Fourier expansion of the exact solution $u(x,t)$ of the problem \eqref{NKGE1d} as 
\begin{align}
    u(x,t) = \sum_{l \in \mathbb{Z}} \widehat{u}_l(t) e^{i \mu_l (x-a)}, \quad x \in\overline{\Omega}, \quad t \ge 0, \label{lm2u}
\end{align}
then we have
\begin{align}  
\widehat{u}_l(t_{n+1})= &\mathrm{cos}(\omega_l \tau) \widehat{u}_l (t_n)  + \displaystyle \frac{\mathrm{sin}(\omega_l \tau)}{\omega_l} \widehat{u}_l'(t_n) - \int_0^{\tau} \frac{\mathrm{sin}(\omega_l (\tau - \theta))}{\varepsilon^2 \omega_l} \left[ e^{i \theta/\varepsilon^2} 
\widehat{(g^n)}_l (\theta)  \right. \nonumber \\
& \left.+  e^{-i\theta/\varepsilon^2}  \widehat{(\overline{g^n})}_l (\theta)  + e^{3i \theta/ \varepsilon^2} \widehat{(h^n)}_l (\theta) + e^{-3i\theta/\varepsilon^2}  \widehat{(\overline{h^n})}_l (\theta) + \widehat{(f^n)}_l (\theta) \right] d \theta, \label{lm2Fu} \\
\widehat{u}_l'(t_{n+1})=&-\omega_l \mathrm{sin}(\omega_l \tau) \widehat{u}_l (t_n)  + \mathrm{cos}(\omega_l \tau) \widehat{u}_l'(t_n) - \int_0^{\tau} \frac{\mathrm{cos}(\omega_l (\tau - \theta))}{\varepsilon^2 } \left[ 
    e^{i \theta/\varepsilon^2} \widehat{(g^n)}_l (\theta) \right. \nonumber \\
    & \left. +  e^{-i\theta/\varepsilon^2}  \widehat{(\overline{g^n})}_l (\theta)  + e^{3i \theta/\varepsilon^2} \widehat{(h^n)}_l (\theta) + e^{-3i\theta/\varepsilon^2}  \widehat{(\overline{h^n})}_l (\theta) + \widehat{(f^n)}_l (\theta) \right] d \theta. \label{lm2Fu'}
\end{align}
\end{lemma}

\begin{proof}
    Recalling the problem 
\eqref{NKGE1d} on the interval $[t_n, t_{n+1}]$ and taking the Fourier transform, we have
\begin{align*}
    \varepsilon^2 \widehat{u}_l''(t_n+s) + \varepsilon^2 \omega_l^2 \widehat{u}_l (t_n+s) + \lambda \widehat{(u^3)}_l (t_n+s) = 0, \quad 0\le s\le \tau.
\end{align*}
Applying Duhamel's principle, we obtain for $0 \le s \le \tau$
\begin{align}
\widehat{u}_l(t_n + s)= & \mathrm{cos}(\omega_l s) \widehat{u}_l (t_n) + \frac{\mathrm{sin}(\omega_l s)}{\omega_l} \widehat{u}_l'(t_n) \nonumber\\
& -\lambda \int_0^s \frac{\mathrm{sin}(\omega_l(s- \theta))}{\varepsilon^2 \omega_l} \widehat{(u^3)}_l(t_n + \theta) d \theta.
\label{lm2.1}
\end{align}
Noticing \eqref{2.2}, \eqref{ansatz1D} and \eqref{gvnhvn}, we have 
\begin{align}
    \lambda \left(u(x,t_n + s) \right)^3 = \left( e^{is/ \varepsilon^2} g^n + e^{3is / \varepsilon^2} h^n + {\rm c.c.} \right) + f^n,
    \label{lm2.2}
\end{align}
where $g^n$, $h^n$ and $f^n$ are given in \eqref{gnhnfn}-\eqref{gvnhvn}.

Plugging \eqref{lm2.2} into \eqref{lm2.1}, we obtain \eqref{lm2Fu}. Differentiating \eqref{lm2.1} with respect to $s$, we can get \eqref{lm2Fu'} in a similar manner and thus the proof is completed.
\end{proof}

\begin{lemma} \label{lemma4.3}
    (New formula of the MTI-FP) For $n \ge 0$, expanding $( I_N \mathbf{u}^n ) (x)$ and $( I_N \dot{\mathbf{u}}^n ) (x)$ in \eqref{4.2} into Fourier series as
    \begin{equation}
\begin{aligned}
& ( I_N \mathbf{u}^n ) (x) = \sum_{l = -N/2}^{N/2-1} \widetilde{(\mathbf{u}^n)}_l e^{i \mu_l (x-a)},\\
&( I_N \dot{\mathbf{u}}^n ) (x) = \sum_{l = -N/2}^{N/2-1} \widetilde{(\dot{\mathbf{u}}^n)}_l e^{i \mu_l (x-a)}, \ x \in \overline{\Omega},
\end{aligned}
\label{lm3u1}
\end{equation}
then we have
\begin{align}
& \widetilde{(\mathbf{u}^{n+1})}_l = \mathrm{cos}(\omega_l \tau) \widetilde{(\mathbf{u}^n)}_l + \displaystyle \frac{\mathrm{sin}(\omega_l \tau)}{\omega_l}\widetilde{(\dot{\mathbf{u}}^{n})}_l - \widetilde{G^n_l},\quad l = \displaystyle - \frac{N}{2},...,\frac{N}{2}-1, \label{lm3u} \\
&\widetilde{(\dot{\mathbf{u}}^{n+1})}_l = - \omega_l \mathrm{sin}(\omega_l \tau) \widetilde{(\mathbf{u}^n)}_l + \mathrm{cos}(\omega_l \tau) \widetilde{(\dot{\mathbf{u}}^{n})}_l - \widetilde{\dot{G}^n_l},
\label{lm3du}
\end{align}
where
\begin{subequations}\label{lm3G}
\begin{align}
\widetilde{G^n_l}  = & e^{i \tau / \varepsilon^2}  c_l(\tau) \widetilde{(\mathbf{g}^{n,0})}_l   + e^{-i \tau / \varepsilon^2}  \overline{c_l}(\tau) \widetilde{(\overline{\mathbf{g}^{n,0}})}_l + p_l(\tau) \widetilde{(\mathbf{h}^{n,0})}_l  +  \overline{p_l}(\tau) \widetilde{( \overline{\mathbf{h}^{n,0}})}_l, \label{lm3G1} \\
\widetilde{\dot{G}^n_l}  = & e^{i \tau / \varepsilon^2} \left[ c_l'(\tau) + \displaystyle \frac{i}{\varepsilon^2} c_l (\tau) \right] \widetilde{(\mathbf{g}^{n,0})}_l  + e^{-i \tau / \varepsilon^2} \left[ \overline{c_l'}(\tau) - \displaystyle \frac{i}{\varepsilon^2} \overline{c_l} (\tau) \right] \widetilde{
(\overline{\mathbf{g}^{n,0}})}_l  \nonumber \\
& + p_l'(\tau) \widetilde{(\mathbf{h}^{n,0})}_l  + \overline{p_l'}(\tau) \widetilde{( \overline{\mathbf{h}^{n,0}})}_l  + \displaystyle \frac{\tau}{2 \varepsilon^2} \widetilde{(\mathbf{f}^{n,1})}_l. \label{lm3G2}
\end{align}
\end{subequations}
\end{lemma}
\begin{proof}
Recalling the MTI-FP method \eqref{MTIu}-\eqref{MTIvr1}, we have
\begin{align}
\widetilde{(\mathbf{u}^{n+1})}_l = e^{i\tau / \varepsilon^2}\widetilde{(\mathbf{v}^{n,1})}_l + e^{-i\tau / \varepsilon^2} \widetilde{(\overline{\mathbf{v}^{n,1}})}_l + \widetilde{(\mathbf{r}^{n,1})}_l,\quad l = -\displaystyle \frac{N}{2},...,\frac{N}{2} - 1. \label{lm3.1}
\end{align}
Plugging the formula \eqref{MTIvr2} for $\widetilde{(\mathbf{v}^{n,1})}_l$ and $\widetilde{(\mathbf{r}^{n,1})}_l$ and the initial data \eqref{MTIIni1} into \eqref{lm3.1}, we can obtain for $l = -\frac{N}{2},..,\frac{N}{2} -1$
\begin{align}
\widetilde{(\mathbf{u}^{n+1})}_l & = \mathrm{Re}\left\{  e^{i \tau / \varepsilon^2} a_l (\tau)\right\} \widetilde{(\mathbf{u}^{n})}_l + \varepsilon^2 \mathrm{Im} \left\{ e^{i \tau / \varepsilon^2} a_l(\tau) \right\} \widetilde{(\dot{\mathbf{u}}^n)}_l   - \widetilde{G_l^n}, \label{lm3.2}
\end{align}
where ${\rm Re}(f)$ and ${\rm Im}(f)$ represent the real part and imaginary part of $f$, respectively.

Thus we can obtain \eqref{lm3u} by applying \eqref{Coeffab}. The proof for \eqref{lm3du} and \eqref{lm3G2} follows similarly with details omitted here for brevity.   
\end{proof}

Introduce local truncation error functions
\begin{align}
    \xi^n(x) = \sum_{l = -N / 2}^{N/2 - 1} \widehat{\xi^n_l}e^{i \mu_l (x - a)}, \quad \dot{\xi}^n(x) =  \sum_{l = -N / 2}^{N/2 - 1} \widehat{\dot{\xi}^n_l}e^{i \mu_l (x - a)}, \quad x \in \overline{\Omega}, 
    \label{loltrun}
\end{align}
where
\begin{align}
& \widehat{\xi_l^n} : = \widehat{u}_l(t_{n+1})
 - \left[  \mathrm{cos}(\omega_l \tau) \widehat{u}_l(t_n) + \displaystyle \frac{\mathrm{sin}(\omega_l \tau)}{\omega_l} \widehat{u}_l'(t_n) - \widehat{\mathcal{G}_l^n} \right], \label{lol}\\
& \widehat{\dot{\xi}_l^n} : = \widehat{u}_l'(t_{n+1})
 - \left[ - \omega_l \mathrm{sin}(\omega_l \tau) \widehat{u}_l(t_n) + \mathrm{cos}(\omega_l \tau) \widehat{u}_l'(t_n) - \widehat{\dot{\mathcal{G}}_l^n} \right],
\label{dlol}
\end{align}
with
\begin{subequations}\label{lolg}
\begin{align}
\widehat{\mathcal{G}_l^n}  := & e^{i \tau / \varepsilon^2} c_l(\tau) \widehat{(g^n)}_l (0)  + e^{-i \tau / \varepsilon^2}  \overline{c_l}(\tau) \widehat{(\overline{g^n})}_l (0)  \nonumber \\
& + \ p_l(\tau) \widehat{(h^n)}_l(0)  + \overline{p_l}(\tau) \widehat{( \overline{h^n})}_l(0), \label{lg}\\       
\widehat{\dot{\mathcal{G}}_l^n} := &      e^{i \tau / \varepsilon^2} \left[ c_l'(\tau) + \displaystyle \frac{i}{\varepsilon^2} c_l (\tau) \right] \widehat{(g^n)}_l(0)  +  \ e^{-i \tau / \varepsilon^2} \left[ \overline{c_l'}(\tau) - \displaystyle \frac{i}{\varepsilon^2} \overline{c_l} (\tau) \right] \widehat{(\overline{g^n})}_l (0)  \nonumber \\ 
& + \ p_l'(\tau) \widehat{(h^n)}_l(0)  + \overline{p_l'}(\tau) \widehat{( \overline{h^n})}_l(0)  + \displaystyle \frac{\tau}{2 \varepsilon^2} \widehat{(f^n)}_l (\tau). \label{ldg}
\end{align}
\end{subequations}
In addition, define the errors from the nonlinear terms as 
\begin{align}
    \eta^n (x) : = \sum_{l = -N /2}^{N/2 - 1} \widehat{\eta^n_l} e^{i \mu_l (x-a)}, \quad \dot{\eta}^n (x) : = \sum_{l = -N /2}^{N/2 - 1} \widehat{\dot{\eta}^n_l} e^{i \mu_l (x-a)},\quad x \in \overline{\Omega}, \label{Nonerr1}
\end{align}
where
\begin{align}
\widehat{\eta^n_l} := \widehat{\mathcal{G}_l^n} - \widetilde{G_l^n}, \qquad \widehat{\dot{\eta}_l^n} : = \widehat{\dot{\mathcal{G}}_l^n} - \widetilde{\dot{G}_l^n}. \label{Nonerr2}
\end{align}
Subtracting the exact flow \eqref{lm2Fu} and \eqref{lm2Fu'} from the numerical flow \eqref{lm3u} and \eqref{lm3du}, we have the following error equations (in phase space) 
\begin{align}
& \widehat{(e^{n+1}_N)}_l = \mathrm{cos}(\omega_l \tau) \widehat{(e^n_N)}_l + \displaystyle \frac{\mathrm{sin}(\omega_l \tau)}{\omega_l}  \widehat{( \dot{e}^n_N)}_l + \widehat{\xi_l^n}- \widehat{\eta_l^n}, \quad l = -\frac{N}{2},...,\frac{N}{2}-1,\label{errorequ1} \\
     & \widehat{(\dot{e}^{n+1}_N)}_l = -\omega_l \mathrm{sin}(\omega_l \tau) \widehat{(e^n_N)}_l + \mathrm{cos}(\omega_l \tau) \widehat{(\dot{e}_N^n)}_l + \widehat{\dot{\xi}_l^n}- \widehat{\dot{\eta}_l^n}.\label{errorequ2}
 \end{align}

\subsection{Energy estimates for error functions}

For the local truncation error functions in \eqref{loltrun}, we have the following estimates.

\begin{lemma}\label{lemma4.4}(Local truncation error) Under the assumption (A), when $0<\tau \le \tau_1$, we have two independent estimates for $0 < \varepsilon \le 1$
    \begin{align}
        \mathcal{E}(\xi^n,\dot{\xi}^n) \lesssim \displaystyle \frac{\tau^6}{\varepsilon^6}, \quad \mathcal{E}(\xi^n,\dot{\xi}^n ) \lesssim \tau^2 \varepsilon
        ^2 + \displaystyle \frac{\tau^4}{\varepsilon^2},\qquad n = 0,1,...,\frac{T}{\tau} - 1. \label{lm4lol}
    \end{align}
\end{lemma}
\begin{proof}
    Noting $e^{i\tau / \varepsilon^2} b_l(\tau - \theta) = \frac{\mathrm{sin}(\omega_l ( \tau - \theta))}{\varepsilon^2 \omega_l} e^{i\theta / \varepsilon^2}$ and \eqref{Coeffcp}, plugging the formula of the exact solution \eqref{lm2Fu} into \eqref{lol} and applying Taylor's expansion, we get
\begin{align}
    \widehat{ \xi^n_l}  = & - \int_0^{\tau} \displaystyle \frac{\mathrm{sin}(\omega_l (\tau - \theta))}{\varepsilon^2 \omega_l} \theta^2 \left[  e^{i \theta / \varepsilon^2} \int_0^1 \widehat{(g^n)}_l ''(\theta \rho)(1 - \rho) d \rho  \right.\nonumber \\
    & \left. + e^{-i \theta / \varepsilon^2} \int_0^1 \widehat{(\overline{g^n})}_l ''(\theta \rho)(1 - \rho) d \rho + e^{3i \theta / \varepsilon^2} \int_0^1 \widehat{(h^n)}_l ''(\theta \rho)(1 - \rho) d \rho 
    \right. \nonumber\\
    & \left. + e^{-3i \theta / \varepsilon^2} \int_0^1 \widehat{(\overline{h^n})}_l ''(\theta \rho)(1 - \rho) d \rho \right] d\theta - \int_0^{\tau} \displaystyle \frac{\mathrm{sin}(\omega_l(\tau - \theta))}{\varepsilon^2 \omega_l} \widehat{(f^n)}_l (\theta) d \theta. \label{lm4.1}
\end{align}
Noting $\left. f^n\right|_{s=0} = 0$, we apply the error estimates of the trapezoidal rule and get
\begin{align*}
\displaystyle \left| \int_0^{\tau} \frac{\mathrm{sin}(\omega_l(\tau - \theta))}{\varepsilon^2 \omega_l} \widehat{(f^n)}_l (\theta) d \theta \right| \lesssim \int_0^{\tau} \frac{\theta (\tau - \theta)}{\varepsilon^2 \omega_l} \left| \frac{d^2 }{d \theta^2} \left[ \mathrm{sin}(\omega_l ( \tau - \theta)) \widehat{(f^n)}_l (\theta) \right] \right| d \theta .
\end{align*}
Thus we can obtain the following bound of \eqref{lm4.1}
\begin{align}
    \left| \widehat{\xi_l^n} \right|\lesssim&\displaystyle \frac{\tau^2}{\varepsilon^2 \omega_l} \int_0^{\tau} \int_0^1 \left[\left| \widehat{(g^n)}_l ''(\theta \rho) \right| + \left| \widehat{(\overline{g^n})}_l ''(\theta \rho) \right| + \left| \widehat{(h^n)}_l ''(\theta \rho) \right| + \left| \widehat{(\overline{h^n})}_l ''(\theta \rho) \right|\right] d \rho d \theta \nonumber \\
    & \ +  \displaystyle \frac{\tau^2}{\varepsilon^2 \omega_l} \int_0^{\tau}\left[ \left| \widehat{(f^n)}_l''(\theta) \right| + \left| \omega_l \widehat{(f^n)}_l'(\theta) \right|  + \left| \omega_l^2 \widehat{(f^n)}_l(\theta)  \right|\right] d \theta. \label{lm4.2}
\end{align}
Noting $\frac{1}{\varepsilon^2 \omega_l} \le 1$ for $l = -N/2,...,N/2 -1$, we get from \eqref{lm4.2} that for $0 < \tau \le \tau_1$,
\begin{align}
    \left\| \xi^n \right\|_{H^1}^2 \lesssim & \tau^6 \left[ \left\| \partial_{ss} g^n \right\|_{L^{\infty}([0,\tau];H^1)}^2 + \left\| \partial_{ss} h^n \right\|_{L^{\infty}([0,\tau];H^1)}^2 + \left\| \partial_{ss} f^n \right\|_{L^{\infty}([0,\tau];H^1)}^2 \displaystyle \right.\nonumber \\
    & \left. + \displaystyle \frac{1}{\varepsilon^4} \left\| \partial_{s} f^n \right\|_{L^{\infty}([0,\tau];H^2)}^2 + \displaystyle \frac{1}{\varepsilon^8} \left\|  f^n \right\|_{L^{\infty}([0,\tau];H^3)}^2 \right]. \label{lm4.3}
\end{align}

By the definition \eqref{gnhnfn} and \eqref{2.2}, adopting the Sobolev embedding theorem \citep{evans2022partial} and applying the prior estimates \eqref{lm1v} and \eqref{lm1r} in Lemma \ref{lemma4.1}, we obtain
\begin{align*}
    \left\| f^n \right\|_{H^3} & \lesssim \left\|  ( r^n)^3 + 3  (r^n)^2 \left(e^{is / \varepsilon^2}v^n +e^{-is / \varepsilon^2} \overline{v^n} \right) + 3 r^n \left( e^{is / \varepsilon^2}v^n + e^{-is / \varepsilon^2}\overline{v^n} \right)^2 \right\|_{H^3} \nonumber \\
    & \lesssim \left\| r^n \right\|_{H^3}^3 + \left\| r^n \right\|_{H^3}^2 \left\| v^n \right\|_{H^3} + \left\| r^n \right\|_{H^3} \left\| v^n \right\|_{H^3}^2 \lesssim \varepsilon^2,\quad 0 \le s \le \tau,
\end{align*}
which implies 
\begin{align} \label{est_fn}
    \left\| f^n \right\|_{L^{\infty}([0,\tau];H^3)} \lesssim \varepsilon^2.
\end{align}
Similarly, we get the estimates for $g^n$, $h^n$ and $\partial_{s} f^n$
\begin{equation}
\begin{aligned}
  &  \left\| \partial_{s}^k g^n \right\|_{L^{\infty}([0,\tau];H^1)}  +   \left\| \partial_{s}^k h^n \right\|_{L^{\infty}([0,\tau];H^1)} \lesssim \varepsilon^{2-2k}, \\
  &\left\| \partial_{s} f^n \right\|_{L^{\infty}([0,\tau];H^2)} \lesssim 1,\qquad k = 1,2. 
  \end{aligned}\label{est_others}
\end{equation}

\noindent Based on the definition \eqref{gnhnfn} and \eqref{2.2}, we have the following estimate of $\partial_{ss} f^n$ for $0 \le s \le \tau$
\begin{align*}
    &\left\| \partial_{ss}f^n  \right\|_{H^1} \\
    &\lesssim \left\| \partial_{ss} \left[ ( r^n)^3 + 3  (r^n)^2 \left(e^{is / \varepsilon^2}v^n +e^{-is / \varepsilon^2} \overline{v^n} \right) + 3 r^n \left( e^{is / \varepsilon^2}v^n + e^{-is / \varepsilon^2}\overline{v^n} \right)^2 \right] \right\|_{H^1}.
\end{align*}
By the Sobolev embedding theorem \citep{evans2022partial} and the prior estimates \eqref{lm1v} and \eqref{lm1r}, we get
\begin{align*}
\left\| \partial_{ss} (r^n)^3 \right\|_{H^1} & \lesssim \left\| (r^n)^2 \partial_{ss} r^n \right\|_{H^1} + \left\| r^n (\partial_s r^n )^2 \right\|_{H^1} \\
& \lesssim \left\| (r^n)^2 \partial_{ss} r^n \right\|_{L^2} + \left\| \partial_x ( (r^n)^2 \partial_{ss} r^n ) \right\|_{L^2} + \left\| r^n (\partial_s r^n )^2 \right\|_{H^1} \\
 & \lesssim \left\| r^n \right\|_{L^{\infty}}^2 \left\| \partial_{ss}   r^n \right\|_{L^2} + \left\| \partial_x  (r^n)^2  \right\|_{L^\infty} \left\| \partial_{ss} r^n \right\|_{L^2} \\
 & \quad +  \left\| r^n \right\|_{L^{\infty}}^2 \left\| \partial_x (\partial_{ss} r^n) \right\|_{L^2} + \left\| r^n (\partial_s r^n )^2 \right\|_{H^1} \\
    & \lesssim \left\| r^n \right\|_{L^{\infty}([0,\tau];H^3)}^2 \left\| \partial_{ss} r^n \right\|_{L^{\infty}([0,\tau];H^1)} \\
    & \quad + \left\| r^n \right\|_{L^{\infty}([0,\tau],H^3)} \left\| \partial_s r^n \right\|_{L^{\infty}([0,\tau],H^3)}^2 
    \lesssim \varepsilon^2.
\end{align*}
Processing other terms similarly, we obtain
\begin{align}
    \left\| \partial_{ss} f^n \right\|_{L^{\infty}([0,\tau];H^1)} \lesssim \frac{1}{\varepsilon^2}. \label{ddfn}
\end{align}
Plugging \eqref{est_fn}-\eqref{ddfn} into \eqref{lm4.3}, we get 
\begin{align}
    \left\| \xi^n \right\|_{H^1}^2 \lesssim \displaystyle \frac{\tau^6}{\varepsilon^4}.
\end{align}
Similarly, by noting $\frac{\left| \mu_l \right|}{\varepsilon^2 \omega_l} \le \frac{1}{\varepsilon}$, we have
\begin{align*}
    \left\| \partial_x\xi^n \right\|_{H^1}^2 \lesssim \displaystyle \frac{\tau^6}{\varepsilon^6},\qquad \left\| \dot{\xi}^n \right\|_{H^1}^2 \lesssim \displaystyle \frac{\tau^6}{\varepsilon^8},
\end{align*}
which completes the first estimate in \eqref{lm4lol}.

On the other hand, by applying Taylor's expansion and truncating at the first order term, we get
\begin{align*}
\widehat{\xi^n_l}= & - \int_0^{\tau} \displaystyle \frac{\mathrm{sin}(\omega_l (\tau - \theta))}{\varepsilon^2 \omega_l} \theta e^{i \theta / \varepsilon^2} \int_0^1 \widehat{(g^n)}_l'(\theta s) ds d \theta  \\
& - \int_0^{\tau} \displaystyle \frac{\mathrm{sin}(\omega_l (\tau - \theta))}{\varepsilon^2 \omega_l} \theta e^{-i \theta / \varepsilon^2} \int_0^1 \widehat{(\overline{g^n})}_l'(\theta s) ds d \theta  \\
&- \int_0^{\tau} \displaystyle \frac{\mathrm{sin}(\omega_l (\tau - \theta))}{\varepsilon^2 \omega_l} \theta e^{3i \theta / \varepsilon^2} \int_0^1 \widehat{(h^n)}_l'(\theta s) ds d \theta \\ 
& - \int_0^{\tau} \displaystyle \frac{\mathrm{sin}(\omega_l (\tau - \theta))}{\varepsilon^2 \omega_l} \theta e^{-3i \theta / \varepsilon^2} \int_0^1 \widehat{(\overline{h^n})}_l'(\theta s) ds d \theta  \\
&-\int_0^{\tau} \displaystyle \frac{\mathrm{sin}(\omega_l (\tau - \theta))}{\varepsilon^2 \omega_l} \widehat{(f^n)}_l (\theta) d \theta. 
\end{align*}
Thus we can obtain the following bound of the above equation
\begin{align}
    \left| \widehat{\xi^n_l} \right| \le & \frac{\tau}{\varepsilon^2 \omega_l} \left[ \int_0^{\tau} \int_0^1 \left| \widehat{(g^n)}_l'(\theta s) \right| ds d \theta  
     + \int_0^{\tau} \int_0^1 \left| \widehat{(\overline{g^n})}_l'(\theta s) \right| ds d \theta \right. \nonumber \\
     & \left. +  \int_0^{\tau}  \int_0^1 \left| \widehat{(h^n)}_l'(\theta s) \right| ds d \theta   + \int_0^{\tau} \int_0^1 \left| \widehat{(\overline{h^n})}_l'(\theta s) \right| ds d \theta \right] \nonumber \\
     & + \frac{1}{\varepsilon^2 \omega_l} \int_0^{\tau} \left| \widehat{(f^n)}_l (\theta) \right| d \theta.\label{lm4.4}
\end{align}
By applying the Cauchy–Schwarz inequality, i.e.
\begin{align*}
& \begin{aligned}
 \left( \int_0^\tau \int_0^1 \left| \widehat{(g^n)}_l'(\theta s) \right| ds d \theta \right)^2 & \le \int_0^{\tau} 1 d\theta \cdot \int_0^{\tau} \left( \int_0^1 \left| \widehat{(g^n)}_l'(\theta s) \right| ds \right)^2 d \theta \\
 & \le \tau \int_0^{\tau} \int_0^1  \left| \widehat{(g^n)}_l'(\theta s) \right|^2 ds d \theta, 
 \end{aligned}\\
& \left( \int_0^{\tau}  \left| \widehat{(f^n)}_l(\theta)  \right| d\theta \right)^2  \le \int_0^\tau 1 d \theta \cdot  \int_0^\tau  \left| \widehat{(f^n)}_l(\theta) \right|^2 d \theta \le \tau \int_0^\tau  \left| \widehat{(f^n)}_l(\theta) \right|^2 d \theta,
\end{align*}
noticing $\varepsilon^2 \omega_l \ge 1$, we use \eqref{est_fn}-\eqref{est_others} and get from \eqref{lm4.4} that
\begin{align}
   \left\| \xi^n \right\|_{H^1}^2 \lesssim &  \tau^4 \left[  \left\| \partial_s g^n \right\|_{L^{\infty}([0,\tau];H^1)}^2  + \left\| \partial_s h^n \right\|_{L^{\infty}([0,\tau];H^1)}^2 \right] + \tau^2 \left\| f^n \right\|_{L^{\infty}([0,\tau];H^1)}^2 \nonumber \\
    \lesssim & \tau^4 +   \tau^2 \varepsilon^4. \label{lm4.5}
\end{align}
Similarly, we can obtain
\begin{align*}
    \left\| \partial_x \xi^n \right\|_{H^1}^2 \lesssim \tau^2 \varepsilon^2 +  \displaystyle \frac{\tau^4}{\varepsilon^2},\qquad \left\| \dot{\xi}^n \right\|_{H^1}^2 \lesssim \tau^2 + \displaystyle \frac{\tau^4}{\varepsilon^4}.
\end{align*}
Thus with the definition of the energy functional \eqref{enefun}, we obtain the two independent estimates in \eqref{lm4lol}. 
\end{proof}

Similarly, for the error functions in \eqref{Nonerr1}, we have the following estimate.
\begin{lemma}\label{lemma4.5}(Nonlinear terms error) Under the assumption (A) and assuming \eqref{thm1err2} holds for $n$ (which will be proved by induction later), we have for any $0 < \tau \le \tau_1$,
    \begin{align}
        \mathcal{E}(\eta^n, \dot{\eta}^n) \lesssim \tau^2 \mathcal{E}(e_N^n,\dot{e}_N^n) + \displaystyle \frac{\tau^2 h^{2 m_0 }}{\varepsilon^2}, \quad n = 0,1,...,\displaystyle \frac{T}{\tau} -1. \label{lm5result}
    \end{align}
\end{lemma}
\begin{proof}
 Define
\begin{subequations}\label{lm5.1}
\begin{align}
&e_{g}^n (x) = g(v^n(x,0)) - (I_N \mathbf{g}^{n,0})(x),  \quad
e_{h}^n (x) = h(v^n(x,0)) - (I_N \mathbf{h}^{n,0})(x),\\
& e_{v}^n (x) = v^n(x,0) - (I_N \mathbf{v}^{n,0})(x), \qquad x \in \overline{\Omega}.   
\end{align}
\end{subequations}
Denote
\begin{align*}
\mathbf{e}_g^n = ( e^n_{g,0},e^n_{g,1},\ldots,e^n_{g,N})^T, \
\mathbf{e}_h^n = ( e^n_{h,0},e^n_{h,1},\ldots,e^n_{h,N})^T,\ 
\mathbf{e}_v^n = ( e^n_{v,0},e^n_{v,1},\ldots,e^n_{v,N})^T,
\end{align*}
where for $j=0,1,\ldots, N$
\[
e^n_{g,j} =  g(v^n(x_j,0)) - g^{n,0}_j,\quad e^n_{h,j} =  h(v^n(x_j,0)) - h^{n,0}_j,\quad
e^n_{v,j} =  v^n(x_j,0) - v^{n,0}_j. 
\]
Noticing $ \left| c_l (\tau) \right| + \left| p_l (\tau) \right| \lesssim \tau $ for $ l = -\frac{N}{2},...,\frac{N}{2}-1$, we obtain from \eqref{lm3G1} and \eqref{lg} that 
\begin{align}
   \left\|  \eta^n \right\|_{H^1}^2  \lesssim & \tau^2 \left[ \left\|  P_N e_g^n \right\|_{H^1}^2 + \left\| P_N e_h^n \right\|_{H^1}^2 \right] \nonumber \\
    \lesssim & \tau^2 \left[ \left\|  I_N \mathbf{e}_g^n \right\|_{H^1}^2 + \left\| I_N \mathbf{e}_h^n \right\|_{H^1}^2 \right]  + \tau^2 h^{2 m_0}, \qquad \label{lm5.3}
\end{align}
where we apply
\begin{align*}
&\left\| P_N e_g^n \right\|_{H^1} \le \left\| P_N g(v^n(x,0)) - I_N g(v^n(x,0)) \right\|_{H^1} + \left\| I_N \mathbf {e}_g^n \right\|_{H^1} \lesssim \left\|  I_N \mathbf{e}_g^n \right\|_{H^1} + h^{m_0}, \\ & \left\| P_N e_h^n \right\|_{H^1} \le \left\| P_N h(v^n(x,0)) - I_N h(v^n(x,0)) \right\|_{H^1} + \left\| I_N \mathbf {e}_h^n \right\|_{H^1} \lesssim \left\|  I_N \mathbf{e}_h^n \right\|_{H^1} + h^{m_0},
\end{align*}
with $g(v^n(x,0))$, $h(v^n(x,0)) \in H^{m_0 + 1}(\Omega)$ by the definition \eqref{NLSWv1d(b)} and \eqref{gvnhvn} and the assumption (A).

By adopting the interpolation error \citep{bao2014uniformly} and the Sobolev inequality \citep{bao2014uniformly,evans2022partial}, we have
\begin{eqnarray}
   \left\| I_N \mathbf{e}_g^n \right\|_{H^1} \lesssim \left\| I_N \mathbf{e}_v^n \right\|_{H^1}, \quad  \left\| I_N \mathbf{e}_h^n \right\|_{H^1}
   \lesssim \left\| I_N \mathbf{e}_v^n \right\|_{H^1},  \label{lm5.4}
\end{eqnarray}
where we apply the controlled $l^{\infty}$ norm of the numerical solution under the assumption (A) and \eqref{thm1err2}. 

Thus plugging \eqref{lm5.4} into \eqref{lm5.3} and noting the initial data \eqref{NLSWv1d(b)} and \eqref{MTIIni1}, we can obtain
\begin{align*}
    \left\| \eta^n \right\|_{H^1}^2 & \lesssim \tau^2 \left\| I_N \mathbf{e}^n_v \right\|_{H^1}^2 + \tau^2 h^{2 m_0 }  \\
    & \lesssim \tau^2 \left[ \left\| I_N u(\cdot,t_n) - (I_N \mathbf{u}^n)(\cdot) \right\|_{H^1}^2 + \varepsilon^4 \left\| I_N \partial_t u(\cdot,t_n) - (I_N \dot{\mathbf{u}}^n)(\cdot) \right\|_{H^1}^2 + h^{2 m_0} \right]  \\
    & \lesssim  \tau^2 \left[  \left\| e_N^n \right\|_{H^1}^2 + \varepsilon^4 \left\| \dot{e}^n_N \right\|_{H^1}^2 +  h^{2 m_0} \right].
\end{align*}
Similarly, we have
\begin{align*}
    \left\| \partial_x \eta^n \right\|_{H^1}^2 + \varepsilon^2 \left\| \dot{\eta}^n \right\|_{H^1}^2 \lesssim \displaystyle \frac{\tau^2}{\varepsilon^2}  \left[  \left\| e_N^n \right\|_{H^1}^2 + \varepsilon^4 \left\| \dot{e}^n_N \right\|_{H^1}^2 +  h^{2 m_0} \right], 
\end{align*}
which completes the proof. 
\end{proof}

\subsection{Proof for Theorem~{\upshape\ref{tm4.1}} by the energy method}

\begin{proof}
For $n=0$, from the initial data \eqref{init00} and the assumption (A), we have
\begin{align*}
    \left\| e_N^0 \right\|_{H^1} + \varepsilon^2 \left\| \dot{e}_N^0 \right\|_{H^1} = \left\| P_N \phi_1 - I_N \phi_1 \right\|_{H^1} + \left\| P_N \phi_2 - I_N \phi_2 \right\|_{H^1} \lesssim h^{m_0 }.
\end{align*}
In addition, using the triangle inequality, we know that there exists a constant $h_1 > 0 $ independent of $\varepsilon$ such that for $0 < h \le h_1$ and $ \tau >0 $,
\begin{align*}
&\left\| I_N \mathbf{u}^0 \right\|_{H^1} \le  \left\| P_N \phi_1 \right\|_{H^1} + \left\| e_N^0 \right\|_{H^1} \le  C_0 + 1, \\
&\left\| I_N \dot{\mathbf{u}}^0 \right\|_{H^1} \le \displaystyle \frac{\left\| P_N \phi_2 \right\|_{H^1}}{\varepsilon^2} + \left\| \dot{e}_N^0 \right\|_{H^1} \le \displaystyle \frac{C_0 + 1}{\varepsilon^2}.
\end{align*}
Thus \eqref{thm1err1}-\eqref{thm1err2} hold for $n = 0$. 

Now we assume \eqref{thm1err1}-\eqref{thm1err2} with $e^n$ and $\dot{e}^n$ replaced 
by $e_N^n$ and $\dot{e}_N^n$, respectively, 
are valid for $n = 0,1,...,m$, and prove the case $n = m+1$. 
Using the Cauchy's inequality, we obtain from the error equations \eqref{errorequ1}-\eqref{errorequ2} that
\begin{align*}
    & \left| \widehat{(e^{n+1}_N )}_l \right|^2 \le (1 + \tau) \left| \mathrm{cos}(\omega_l \tau) \widehat{(e^n_N)}_l + \displaystyle \frac{\mathrm{sin}(\omega_l \tau)}{\omega_l}  \widehat{( \dot{e}^n_N)}_l \right|^2 + \displaystyle \frac{1 + \tau}{\tau}  \left| \widehat{\xi_l^n}- \widehat{\eta_l^n} \right|^2, \\ 
    & \left| \widehat{(\dot{e}^{n+1}_N)}_l  \right|^2 \le (1 + \tau) \left|  -\omega_l \mathrm{sin}(\omega_l \tau) \widehat{(e^n_N)}_l + \mathrm{cos}(\omega_l \tau) \widehat{(\dot{e}_N^n)}_l \right|^2 + \displaystyle \frac{1 + \tau}{\tau} \left| \widehat{\dot{\xi}_l^n}- \widehat{\dot{\eta}_l^n} \right|^2 . 
\end{align*}
Multiplying the above two equations by $(\mu_l^2 + \frac{1}{\varepsilon^2} )(1 + \mu_l^2)$ and $\varepsilon^2 ( 1 + \mu_l^2)$, respectively, and then summing them up for $l = -\frac{N}{2},...,\frac{N}{2}-1$, we get
\begin{align}
    \mathcal{E}(e_N^{n+1},\dot{e}_N^{n+1}) & \le (1 + \tau) \mathcal{E}(e_N^{n},\dot{e}_N^{n}) + \displaystyle \frac{1 + \tau}{\tau} \mathcal{E}(\xi^n - \eta^n,\dot{\xi}^n - \dot{\eta}^n) \nonumber \\
    & \le (1 + \tau) \mathcal{E}(e_N^{n},\dot{e}_N^{n}) + \displaystyle \frac{1 + \tau}{\tau} \left[ \mathcal{E}(\xi^n,\dot{\xi}^n ) + \mathcal{E}(\eta^n,\dot{\eta}^n ) \right].
    \label{thm1P2}
\end{align}
Plugging the results in Lemma \ref{lemma4.4} and Lemma \ref{lemma4.5} into \eqref{thm1P2}, we obtain two independent estimates
\begin{subequations}\label{thm1P3}
\begin{align}
    \mathcal{E}(e_N^{n+1},\dot{e}_N^{n+1}) - \mathcal{E}(e_N^{n},\dot{e}_N^{n}) & \lesssim   \tau \mathcal{E}(e_N^{n},\dot{e}_N^{n}) + \displaystyle \frac{\tau h^{2 m_0 }}{\varepsilon^2} + \displaystyle \frac{\tau^5}{\varepsilon^6}, \\
    \mathcal{E}(e_N^{n+1},\dot{e}_N^{n+1}) - \mathcal{E}(e_N^{n},\dot{e}_N^{n}) & \lesssim \tau \mathcal{E}(e_N^{n},\dot{e}_N^{n}) + \displaystyle \frac{\tau h^{2 m_0 }}{\varepsilon^2} + \tau \varepsilon^2 + \displaystyle \frac{\tau^3}{\varepsilon^2}.
\end{align}
\end{subequations}
Summing  \eqref{thm1P3} for $n = 0,...,m $ and applying the discrete Gronwall's inequality \citep{tao2006local}, we get
\begin{align*}
    \mathcal{E}(e_N^{m+1},\dot{e}_N^{m+1})  \lesssim \displaystyle \frac{\tau^4}{\varepsilon^6} + \frac{h^{2 m_0 }}{\varepsilon^2}, \qquad 
    \mathcal{E}(e_N^{m+1},\dot{e}_N^{m+1})  \lesssim \displaystyle \frac{\tau^2}{\varepsilon^2} + \varepsilon^2 +  \frac{h^{2 m_0 }}{\varepsilon^2},
\end{align*}
which implies by the definition \eqref{enefun}
\begin{equation}
\begin{aligned}
   & \left\| e^{m+1}_N \right\|_{H^1} + \varepsilon^2 \left\| \dot{e}^{m+1}_N \right\|_{H^1} \lesssim h^{m_0 } + \displaystyle \frac{\tau^2}{\varepsilon^2}, \\
   & \left\| e^{m+1}_N \right\|_{H^1} + \varepsilon^2 \left\| \dot{e}^{m+1}_N \right\|_{H^1} \lesssim h^{m_0 } + \tau + \varepsilon^2.
    \end{aligned}\label{est_eN}
\end{equation}
Combining \eqref{est_eN} and \eqref{Tri_err}, we prove \eqref{thm1err1} for $ n = m + 1$.

\noindent Thus by taking the minimum of \eqref{thm1err1}, we obtain a uniform error bound with respect to $\varepsilon\in(0,1]$
\begin{align*}
    \left\| e^{m+1} \right\|_{H^1} + \varepsilon^2 \left\| \dot{e}^{m+1} \right\|_{H^1} \lesssim h^{m_0} + \tau.
\end{align*}

\noindent As for \eqref{thm1err2}, similarly by the triangle inequality, there exist constants $\tau_2>0$ and $h_2 > 0$ independent of $\varepsilon$ such that for $0 < \tau \le \tau_2$ and $0 < h \le h_2$
\begin{align*}
    & \left\| I_N \mathbf{u}^{m + 1} \right\|_{H^1} \le \left\|  u(t_{m+1}) \right\|_{H^1} + \left\| e^{m+1} \right\|_{H^1} \le C_0 + 1, \\
   & \left\| I_N\dot{\mathbf{u}}^{m + 1} \right\|_{H^1} \le \left\|  \partial_t u(t_{m+1}) \right\|_{H^1} + \left\| \dot{e}^{m+1} \right\|_{H^1} \le \displaystyle \frac{C_0 + 1}{\varepsilon^2}.
\end{align*}

\noindent Thus \eqref{thm1err2} also holds for $n = m + 1$ and the proof is completed by taking $\tau_0 = \min \left\{\tau_1, \tau_2 \right\}$ and $h_0 = \min \left\{ h_1,h_2 \right\}$. 
\end{proof}

\begin{remark}
If the initial data $\phi_1(x)$ and $\phi_2 (x)$ are complex-valued functions, then the solution of the NKGE \eqref{1.1} is complex-valued. In this case, the multiscale decomposition  over the time interval $[t_n,t_{n+1}]$ becomes \citep{bao2014uniformly,machihara2002nonrelativistic,tsutsumi1984nonrelativistic}
\begin{align*}
    u(x,t_n+s) = e^{is / \varepsilon^2} v^n_+(x,s) + e^{-is/\varepsilon^2} \overline{v^n_-(x,s)} + r^n(x,s), \quad 0 \le s \le \tau.
\end{align*}
Then the construction of decomposed equations and the corresponding numerical discretizations proposed
in this paper can be extended 
easily for this case \citep{bao2014uniformly}.
\end{remark}

\begin{remark}
Compared to the MTI-FP method proposed in \citep{bao2014uniformly}, the homogeneous transmission condition $\partial_s v^n(x,0) = 0$ brings two advantages for the MTI-FP method in this work: (i) it relaxes the regularity requirement $u \in C^1([0,T];H^{m_0 + 2})$ in \citep{bao2014uniformly} to $u \in C^1([0,T];H^{m_0 + 1})$, 
and (ii) it helps the simplified MTI-FP method \eqref{MTIu}-\eqref{init00} achieve the optimal spatial accuracy uniformly for $0 < \varepsilon \le 1$.
\end{remark}

\begin{remark}
  When $d = 1$, i.e. one-dimensional case, Theorem \ref{tm4.1} holds without any CFL-conditions. However, for $d = 2$ or $d = 3$, due to the use of inverse inequalities \citep{bao2014uniformly,bao2012analysis,shen2011spectral} \ to control the $l^{\infty}$ norm of the numerical solution, we have to impose the technical condition
  \begin{align*}
      \tau \lesssim \rho_d(h), \quad \mathrm{with} \quad\rho_d(h) = 
      \begin{cases}
          1 / \left| \mathrm{ln} h\right|, \quad & d = 2,\\
          \sqrt{h}, \quad & d = 3.
      \end{cases}
  \end{align*}
Furthermore, if under a stronger assumption (B) of the regularity of $u(x,t)$, i.e. for $m_0 \ge 4$
\begin{align*}
 \text{(B)} \quad  u \in C^1 \left([0,T]; H^{m_0+2}_p (\Omega) \right), \ \left\| u \right\|_{L^{\infty} \left( [0,T]; H^{m_0 +2} \right)} + \varepsilon^2 \left\| \partial_t u \right\|_{L^{\infty} \left( [0,T]; H^{m_0 +2} \right)} \lesssim 1, \nonumber 
\end{align*}
we can derive error estimates in the $H^2$-norm and all of the above analysis could be extended easily. In this case, by using the following Sobolev inequalities \citep{evans2022partial,bao2014uniformly}, no CFL-conditions are needed for $d =1, 2, 3$,
\begin{align*}
    \left\| u \right\|_{L^{\infty}(\Omega)} \le C \left\| u \right\|_{H^2 (\Omega)},\quad 
    \left\| u \right\|_{W^{1,p}(\Omega)} \le C \left\| u \right\|_{H^2 (\Omega)}, \quad 1 < p < 6,
\end{align*}
where $\Omega$ is a bounded domain in $d$-dimensional space $(d = 1 , 2, 3)$.

\end{remark}

\begin{remark}
As established in Section~\ref{sec2}, the transmission condition could be further extended to a generalized function $\gamma(x)$ satisfying certain regularity requirement
\begin{align}
    \partial_s v^n(x,0) = \gamma(x),\quad  \text{with} \quad \gamma(x) \in H^{m_0 + 1}(\Omega).\label{general}
\end{align}
Under the assumption (A), the modified MTI-FP method with \eqref{general} obtains the same error bounds \eqref{thm1err1}-\eqref{thm1err3} in Theorem \ref{tm4.1} and all of the above analysis follows in a similar way. Note that if $\gamma(x) \in H^{m_0}(\Omega)$, the MTI-FP method would face one order reduction in spatial accuracy, i.e.
\begin{align*}
    \left\| e^n \right\|_{H^1} + \varepsilon^2 \left\| \dot{e}^n \right\|_{H^1}  \lesssim h^{m_0 - 1} + \tau, \quad 0 \le n \le \displaystyle \frac{T}{\tau}.
\end{align*}
\end{remark}

\section{A multiscale interpolation in time and its uniformly accurate error bounds}\label{sec5}
In this section, we present a multiscale interpolation in time based on the multiscale decomposition 
\eqref{ansatz1D} and the numerical results obtained via the MTI-FP \eqref{MTIu}-\eqref{init00}.  

\begin{figure}[ht!]
\centering
\includegraphics[width=0.54\textwidth]{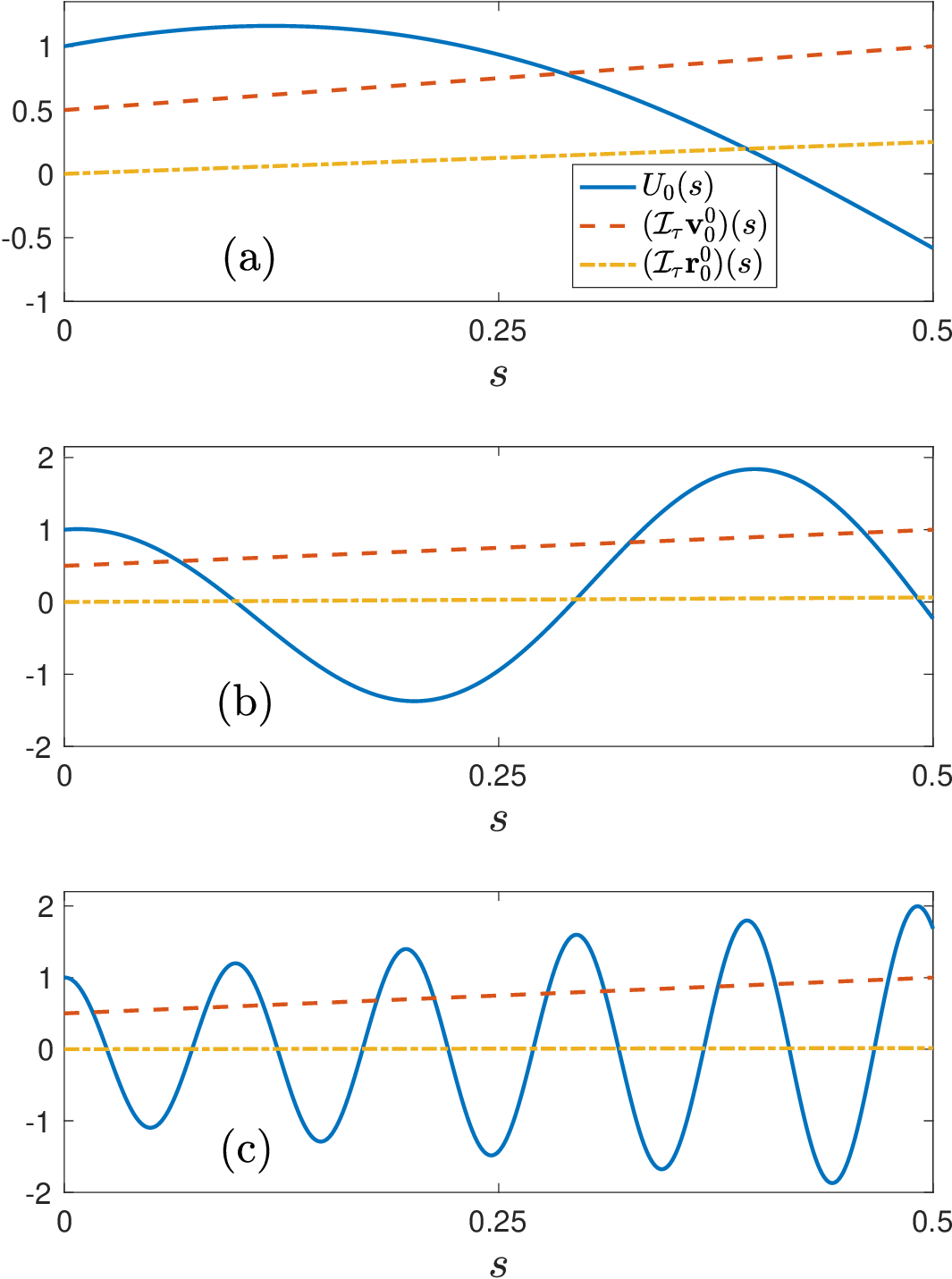}
\caption{Illustration of the multiscale interpolation \eqref{MTIint2} with $\tau=0.5$,
$v_{0}^{0,0}=0.5$, $v_{0}^{0,1}=1.0$ and $r_{0}^{0,1}=\varepsilon^2$ for different $\varepsilon$:
(a) $\varepsilon=0.5$, (b) $\varepsilon=0.25$, and (c) $\varepsilon=0.125$.}
\label{FigureInterpolation}
\end{figure}

\subsection{The multiscale interpolation in time}

Let $\mathcal{I}_\tau: C([0,\tau]) ({\rm or}\ {\mathbb C}^2) \rightarrow W:={\rm span}\{1, s\}$ be the linear 
interpolation operator, i.e.,
\begin{align}
&(\mathcal{I}_\tau \mathbf{w})(s) =\frac{\tau - s}{\tau} w_0 +\frac{s}{\tau}w_1, \quad  0\le s\le \tau, \qquad 
\mathbf{w}=(w_0,w_1)^T\in {\mathbb C}^2,\\ 
&(\mathcal{I}_\tau w)(s) =\frac{\tau - s}{\tau} w(0) +\frac{s}{\tau}w(\tau), \quad  0\le s\le \tau,  \qquad 
w\in  C([0,\tau]).
\end{align}
Based on the numerical solutions from the MTI-FP \eqref{MTIu}-\eqref{init00}, denote
\[\mathbf{v}_j^n=(v^{n,0}_j,v^{n,1}_j)^T\in{\mathbb C}^2, \quad  \mathbf{r}_j^n=(0,r^{n,1}_j)^T\in{\mathbb C}^2,
\quad j=0,1,\ldots,N, \quad n\ge0,\]
then we can present a multiscale interpolation in time for $t\ge0$ as
\begin{equation}
\mathbf{U}(t) = (U_0(t),...,U_N(t))^T,  \quad t \ge 0, \label{MTIint1}
\end{equation}
where for $j=0,1,\ldots,N$
\begin{align}
     U_j(t_n+s)&=e^{is/\varepsilon^2} (\mathcal{I}_\tau\mathbf{v}_j^n)(s)+
     {\rm c.c.} + (\mathcal{I}_\tau\mathbf{r}_j^n)(s), \quad 0 \le s \le \tau, \quad n\ge0,\label{MTIint2}
\end{align}
with
\[ (\mathcal{I}_\tau\mathbf{v}_j^n)(s)=\frac{\tau-s}{\tau}v_j^{n,0}+\frac{s}{\tau}v_j^{n,1},\qquad 
(\mathcal{I}_\tau\mathbf{r}_j^n)(s)=\frac{s}{\tau} r_j^{n,1}, \qquad 0\le s\le \tau.
\]

To illustrate the above multiscale interpolation \eqref{MTIint2} clearly, Figure~\ref{FigureInterpolation}
plots $U_{0}(s)$, $(\mathcal{I}_\tau\mathbf{v}_{0}^0)(s)$ and $(\mathcal{I}_\tau\mathbf{r}_{0}^0)(s)$ with $\tau=0.5$,
$v_{0}^{0,0}=0.5$, $v_{0}^{0,1}=1.0$ and $r_{0}^{0,1}=\varepsilon^2$ for different $\varepsilon$.

\subsection{A uniformly accurate error bound}

For the multiscale interpolation \eqref{MTIint2}, we have the following error bounds.

\begin{theorem}\label{tm2} Under the assumption (A) and $\tau_0$, $h_0$ independent of $\varepsilon$ obtained in Theorem~\ref{tm4.1}, for any $0 < \varepsilon \le 1$, when $0 < h \le h_0$ and $0 < \tau \le \tau_0$, we have for $0 \le t \le T$
\begin{equation}\label{thm2err1}
\begin{aligned}
   & \left\|  u(\cdot,t) - (I_N\mathbf{U})(\cdot,t) \right\|_{H^1}  \lesssim h^{m_0} + \displaystyle \frac{\tau^2}{\varepsilon^2}, \\
   & \left\|  u(\cdot,t) - (I_N\mathbf{U})(\cdot,t)  \right\|_{H^1}  \lesssim h^{m_0} + \tau + \varepsilon^2. 
    \end{aligned}
\end{equation}
Thus, by first taking the minimum of the two error estimates in \eqref{thm2err1} and then taking the maximum for $\varepsilon \in (0,1]$, we obtain a uniform error estimate with respect to  $\varepsilon \in (0,1]$ as
\begin{align}
\left\|  u(\cdot,t) - (I_N\mathbf{U})(\cdot,t) \right\|_{H^1} &  \lesssim h^{m_0} +\max_{0<\varepsilon\le 1}\min\left  \{\frac{\tau^2}{\varepsilon^2}, \varepsilon^2\right\} \nonumber\\
& \lesssim h^{m_0} + \tau,\qquad  0 \le t \le T.  \label{thm2err2}
\end{align}
\end{theorem}

\begin{proof}
Using the triangle inequality, we get
\begin{align}
\left\|  u(\cdot,t) -(I_N\mathbf{U})(\cdot,t)\right\|_{H^1}
&\lesssim\left\| u(\cdot,t) -(P_N u)(\cdot,t)\right\|_{H^1}+\left\| P_N u(\cdot,t) - (I_N\mathbf{U})(\cdot,t) 
\right\|_{H^1}\nonumber \\
&\lesssim h^{m_0} + \left\| (P_N u)(\cdot,t) - (I_N\mathbf{U})(\cdot,t) \right\|_{H^1}, \quad 0\le t \le T.
\end{align}
For $n\ge0$ and $x\in\overline{\Omega}$, define
\begin{equation}\label{thm2P2.3}
u_N^\tau(x,t_n+s):= e^{is / \varepsilon^2} (\mathcal{I}_\tau (P_Nv^n))(x,s)+ {\rm c.c.} 
+ (\mathcal{I}_\tau (P_Nr^n))(x,s), \quad  0\le s\le \tau, 
\end{equation}
then for $t\in[t_n,t_{n+1}]$ (or $s\in[0,\tau]$), noting \eqref{thm2P2.3}, \eqref{MTIint1} and \eqref{MTIint2}, we have
\begin{align}
& \left\| (P_N u)(\cdot,t_n+s) - (I_N\mathbf{U})(\cdot,t_n+s) \right\|_{H^1}\nonumber\\
& \lesssim \left\| (P_N u)(\cdot,t_n+s) -u_N^\tau(\cdot,t_n+s) \right\|_{H^1}+
\left\| u_N^\tau(\cdot,t_n+s) - (I_N\mathbf{U})(\cdot,t_n+s) \right\|_{H^1}\nonumber\\
& \lesssim \left\|  v^n(\cdot,s) - (\mathcal{I}_\tau v^n)(\cdot,s) \right\|_{H^1} + 
\left\|  r^n(\cdot,s) - (\mathcal{I}_\tau r^n)(\cdot,s) \right\|_{H^1}\nonumber\\
& +\left\|(P_Nv^n)(\cdot,0)-(I_N\mathbf{v}^{n,0})(\cdot)\right\|_{H^1}+
\left\|(P_Nv^n)(\cdot,\tau)-(I_N\mathbf{v}^{n,1})(\cdot)\right\|_{H^1}\nonumber\\
& +\left\|(P_Nr^n)(\cdot,\tau)-(I_N\mathbf{r}^{n,1})(\cdot)\right\|_{H^1}.
\label{thm2P10}
\end{align}
Noting the prior estimates \eqref{lm1v}, we obtain for $0\le s\le \tau$
\begin{subequations}\label{thm2P3}
\begin{align}
&\left\| v^n(\cdot, s) - (\mathcal{I}_\tau v^n))(\cdot,s)\right\|_{H^1} \lesssim \tau \left\| \partial_s v^n \right\|_{L^{\infty}([0,\tau];H^1)} \lesssim \tau,  \label{thm2P3.1} \\
& \left\| v^n(\cdot, s) - (\mathcal{I}_\tau v^n)(\cdot,s) \right\|_{H^1} \lesssim \tau^2 \left\| \partial_{ss} v^n \right\|_{L^{\infty}([0,\tau];H^1)} \lesssim \displaystyle \frac{\tau^2}{\varepsilon^2}. \label{thm2P3.2}
\end{align}
\end{subequations}
Similarly, we have 
\begin{equation}
\begin{aligned}
& \left\| r^n(\cdot, s) - (\mathcal{I}_\tau r^n)(\cdot,s) \right\|_{H^1} \lesssim \tau, \\
& \left\| r^n(\cdot, s) -(\mathcal{I}_\tau r^n)(\cdot,s) \right\|_{H^1} \lesssim \displaystyle \frac{\tau^2}{\varepsilon^2}, \quad 0 \le s \le \tau. 
\end{aligned}\label{thm2P4}
\end{equation}
From the initial data \eqref{NLSWv1d(b)}, \eqref{MTIIni1} and the error estimate \eqref{thm1err1}, we immediately obtain
\begin{subequations}\label{Pnvn0}
\begin{align}
&\left\|(P_Nv^n)(\cdot,0)-(I_N\mathbf{v}^{n,0})(\cdot)\right\|_{H^1}\lesssim h^{m_0}+\tau+\varepsilon^2,\\
&\left\|(P_Nv^n)(\cdot,0)-(I_N\mathbf{v}^{n,0})(\cdot)\right\|_{H^1}\lesssim h^{m_0}+\frac{\tau^2}{\varepsilon^2}.
\end{align}
\end{subequations}
Subtracting the Fourier coefficients $\widehat{(v^n)}_l(\tau)$ in \eqref{Sol1.v} from $\widetilde{(\mathbf{v}^{n,1})}_l$ in \eqref{MTIvr2.v}, noting \eqref{thm1err1}, and using  $|a_l(\tau)| \lesssim1$ and $|c_l (\tau)| \lesssim \tau$, we get
\begin{align}
\left\| (P_N v^n)(\cdot,\tau) - (I_N \mathbf{v}^{n,1}) (\cdot) \right\|_{H^1} \lesssim &  \left\| (P_N v^n )(\cdot,0) - ( I_N\mathbf{v}^{n,0}) (\cdot) \right\|_{H^1} \nonumber \\
& + \tau \left( \left\| g^n \right\|_{L^{\infty}([0,\tau];H^1)} + \left\|  I_N \mathbf{g}^{n,0} \right\|_{H^1} \right) \nonumber \\
\lesssim & \left\| e_N^n \right\|_{H^1} + \varepsilon^2 \left\| \dot{e}_N^n \right\|_{H^1} + \tau
 \lesssim   h^{m_0} + \tau + \varepsilon^2.  \label{thm2P6}
\end{align}
On the other hand, similar to estimates of the local truncation errors in Lemma~\ref{lemma4.4}, applying Taylor's expansion and the error bounds \eqref{thm1err1}, we have
\begin{align}
& \left\| (P_N v^n)(\cdot,\tau) - (I_N \mathbf{v}^{n,1}) (\cdot) \right\|_{H^1} \nonumber \\
&\lesssim  \left\| (P_N v^n )(\cdot,0) - ( I_N\mathbf{v}^{n,0}) (\cdot) \right\|_{H^1} + \tau^2  \left\| \partial_s g^n \right\|_{L^{\infty}([0,\tau];H^1)}  \nonumber\\
&\lesssim  h^{m_0} + \tau^2/\varepsilon^2.\label{thm2P7}
\end{align}
Similarly, we can obtain the following estimates
\begin{equation}\label{thm2P9}
\begin{aligned}
&\left\| (P_N r^n)(\cdot,\tau) - (I_N \mathbf{r}^{n,1}) (\cdot) \right\|_{H^1}\lesssim   h^{m_0} + \tau + \varepsilon^2,\\
&\left\| (P_N r^n)(\cdot,\tau) - (I_N \mathbf{r}^{n,1}) (\cdot) \right\|_{H^1}\lesssim
h^{m_0} + \frac{\tau^2}{\varepsilon^2}.
\end{aligned}
\end{equation}
Plugging \eqref{thm2P3}-\eqref{thm2P4}, \eqref{Pnvn0}, \eqref{thm2P6}-\eqref{thm2P7} and \eqref{thm2P9}
into \eqref{thm2P10} and \eqref{thm2P2.3}, we obtain the error estimates in \eqref{thm2err1} and 
\eqref{thm2err2}. The proof is completed.
\end{proof}


\section{Numerical results} \label{sec6}
In this section, we report numerical results to verify our error bounds and to demonstrate their sharpness as well as 
to apply the MTI-FP method for numerically studying convergent rates of the NKGE \eqref{1.1} to its different limiting models.

\subsection{Accuracy test}
We take $d = 1$, $\lambda=1$ and 
\begin{align}
\phi_1(x) = \frac{1}{2}\text{sech}(x^2), \quad \phi_2(x)  = \frac{1}{2} e^{-x^2}, \quad x \in {\mathbb R}, \label{5.1}
\end{align}
in \eqref{1.1}, and $\Omega = (-16,16)$ in \eqref{NKGE1d}, which is large enough such that the truncation error is negligible.

Denote 
\begin{align*}
e^{\varepsilon}_{\tau,h}(t_n) = \left\| u(\cdot,t_n) - I_N \mathbf{u}^{n} \right\|_{H^1}, \quad e^{*}_{\tau,h} (t_n) = \max_{0 < \varepsilon \le 1} \left\{ e^{\varepsilon}_{\tau,h} (t_n) \right\},\quad n\ge0.
\end{align*}

\begin{table}[htbp]
\centering
\caption{\textit{Spatial errors with $\tau = 1 \times 10^{-6}$ for different $\varepsilon$ and $h$}}
\label{t1}
\begin{tabular}{@{}lllll@{}}
\toprule
{$e^{\varepsilon}_{\tau,h}(1)$}\quad\quad \quad \quad \quad &
{$h_0 = 1$} \quad\quad\quad\quad\quad &
{$ h_0 / 2$} \quad\quad\quad\quad\quad &
{$h_0 / 2^2$} \quad\quad\quad\quad\quad &
{$h_0 / 2^3$}\quad \quad \quad \\
\midrule
{$\varepsilon_0=0.5$} & 1.63E-01 &  9.82E-03  & 2.94E-05  & 2.16E-09  \\
{$\varepsilon_0/2^1$} & 1.45E-01 &  1.55E-02  & 6.02E-05 &  7.28E-09 \\
{$\varepsilon_0/2^2$} & 7.72E-02  &  4.26E-03  & 1.89E-05  &  3.70E-09 \\
{$\varepsilon_0 / 2^3$} & 1.48E-01  &  1.26E-02   & 9.49E-05  &  5.02E-09 \\
{$\varepsilon_0 / 2^4$} & 1.09E-01 &  1.19E-02   & 7.63E-05  &  7.24E-09 \\
{$\varepsilon_0 / 2^5$} & 1.59E-01  &  9.71E-03  & 7.54E-05  &  6.70E-09 \\
{$\varepsilon_0 / 2^7$} & 1.67E-01  &  9.15E-03   & 4.89E-05  &  1.49E-08 \\
{$\varepsilon_0 / 2^{9}$} & 3.85E-02  &  1.38E-02 &  8.05E-05  &  1.57E-07 \\
{$\varepsilon_0 / 2^{11}$} & 1.73E-01  & 8.26E-03  &  5.14E-05 &  1.26E-07 \\
{$\varepsilon_0 / 2^{13}$} & 1.44E-01  & 8.21E-03  & 8.75E-05  & 1.30E-07 \\
\bottomrule
\end{tabular}
\end{table}

\begin{table}[htbp]
\centering
\caption{\textit{Temporal errors with $h = 1/32$ for different $\varepsilon$ and $\tau$}}
\label{t2}
\begin{tabular}{@{}llllllll@{}}
\toprule
{$e^{\varepsilon}_{\tau,h}$(1)}\quad\quad &
{$\tau_0=0.2$} \quad\quad&
{$\tau_0/2^2$} \quad\quad&
{$\tau_0/2^4$} \quad\quad&
{$\tau_0/2^6$} \quad\quad&
{$\tau_0/2^8$} \quad\quad&
{$\tau_0/2^{10}$} \quad\quad&
{$\tau_0/2^{12}$} \\
\midrule
{$\varepsilon_0=0.5$} & 1.54E-02 & 9.70E-04 & 6.01E-05 & 3.74E-06 & 2.34E-07 & 1.47E-08 & 2.35E-09 \\[1pt]
rate & {---} & 1.99 & 2.01 & 2.00 & 2.00 & 1.99 & 1.33 \\
\midrule
{$\varepsilon_0/2^1$} & 1.45E-02 & 4.02E-03 & 2.64E-04 & 1.64E-05 & 1.03E-06 & 6.44E-08 & 5.20E-09 \\[1pt]
rate & {---} & 0.92 & 1.96 & 2.00 & 2.00 & 2.00 & 1.82 \\
\midrule
{$\varepsilon_0/2^2$} & 2.50E-02 & 6.80E-03 & 1.25E-03 & 8.04E-05 & 4.99E-06 & 3.11E-07 & 1.93E-08 \\[1pt]
rate & {---} & 0.94 & 1.22 & 1.98 & 2.00 & 2.00 & 2.00 \\
\midrule
{$\varepsilon_0 / 2^3$} & 2.67E-02 & 6.83E-03 & 2.01E-03 & 3.17E-04 & 2.04E-05 & 1.27E-06 & 7.98E-08 \\[1pt]
rate & {---} & 0.98 & 0.88 & 1.33 & 1.98 & 2.00 & 2.00 \\
\midrule
{$\varepsilon_0 / 2^4$} & 2.44E-02 & 6.08E-03 & 1.53E-03 & 4.25E-04 & 6.84E-05 & 4.52E-06 & 2.83E-07 \\[1pt]
rate & {---} & 1.00 & 1.00 & 0.92 & 1.32 & 1.96 & 2.00 \\
\midrule
{$\varepsilon_0 / 2^5$} & 2.55E-02 & 6.32E-03 & 1.57E-03 & 3.97E-04 & 1.11E-04 & 1.75E-05 & 1.15E-06 \\[1pt]
rate & {---} & 1.01 & 1.01 & 0.99 & 0.92 & 1.33 & 1.97 \\
\midrule
{$\varepsilon_0 / 2^7$} & 2.59E-02 & 6.40E-03 & 1.58E-03 & 3.95E-04 & 9.88E-05 & 2.51E-05 & 7.04E-06 \\[1pt]
rate & {---} & 1.01 & 1.01 & 1.00 & 1.00 & 0.99 & 0.92 \\
\midrule
{$\varepsilon_0 / 2^{9}$} & 2.43E-02 & 6.08E-03 & 1.51E-03 & 3.76E-04 & 9.39E-05 & 2.35E-05 & 5.92E-06 \\[1pt]
rate & {---} & 1.00 & 1.01 & 1.00 & 1.00 & 1.00 & 1.00 \\
\midrule
{$\varepsilon_0 / 2^{11}$} & 2.62E-02 & 6.46E-03 & 1.60E-03 & 3.98E-04 & 9.95E-05 & 2.49E-05 & 6.22E-06 \\[1pt]
rate & {---} & 1.01 & 1.01 & 1.00 & 1.00 & 1.00 & 1.00 \\
\midrule
{$\varepsilon_0 / 2^{13}$} & 2.67E-02 & 6.59E-03 & 1.63E-03 & 4.06E-04 & 1.01E-04 & 2.54E-05 & 6.35E-06 \\[1pt]
rate & {---} & 1.01 & 1.01 & 1.00 & 1.00 & 1.00 & 1.00 \\
\hline \midrule 
{$e^{*}_{\tau,h}(1)$} & 2.67E-02 & 6.83E-03 & 2.01E-03 & 4.25E-04 & 1.11E-04 & 2.54E-05 & 7.04E-06 \\[1pt]
rate & {---} & 0.98 & 0.88 & 1.12 & 0.97 & 1.06 & 0.92 \\
\bottomrule
\end{tabular}
\end{table}

Since the exact solution is not available, the ``exact'' solution is computed numerically by the MTI-FP method under well-prepared initial data \citep{bao2014uniformly} with very small mesh size $h = h_e =  1/32$ and time step $\tau = \tau_e = 1 \times 10^{-6}$. Table~\ref{t1} shows the spatial errors of the MTI-FP method \eqref{MTIu}-\eqref{init00} at $t_n=1$ for different $\varepsilon$ and $h$ with a very small time step $\tau = \tau_e$ such that the temporal discretization error is negligible. Table~\ref{t2} shows the temporal errors at $t_n=1$ for different $\varepsilon$ and $\tau$ with a very fine mesh $h = h_e $ such that the spatial discretization error  is negligible. In addition, Figure~\ref{Figuresup2} plots
the multiscale interpolation errors $|u(0,t)-U_{N/2}(t)|$ with $U_{N/2}(t)$ obtained via the interpolation \eqref{MTIint1}-\eqref{MTIint2} for different $\varepsilon$ and $\tau$ with $h=h_e$ (or $N = 2^{10}$). 

\begin{figure}[htbp]
\centering
\includegraphics[width=0.75\textwidth]{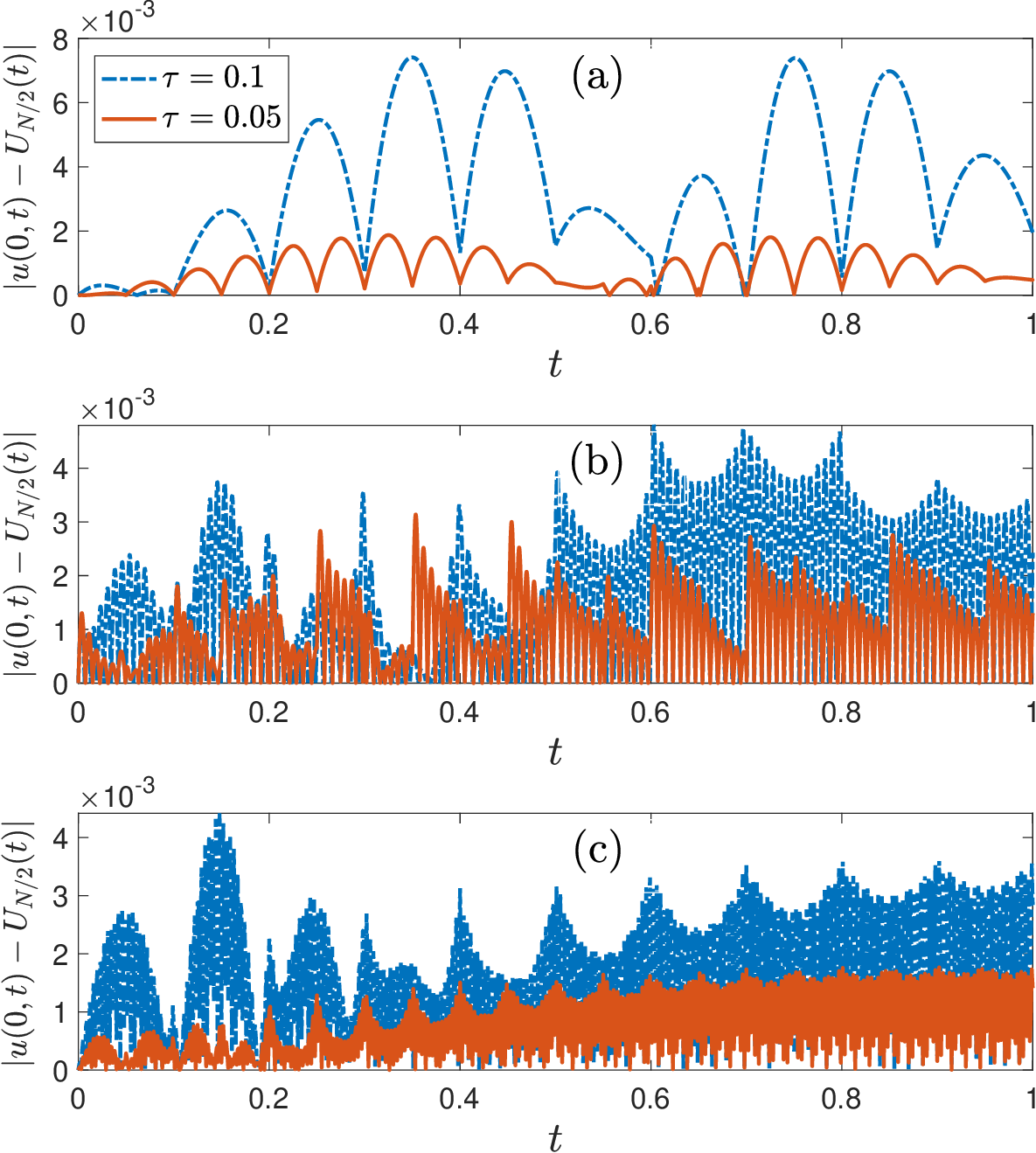}
\caption{Plots of the multiscale interpolation error $|u(0,t)-U_{N/2}(t)|$ with $U_{N/2}(t)$ obtained via the multiscale interpolation \eqref{MTIint1}-\eqref{MTIint2} for different $\varepsilon$ and $\tau$ with $N=2^{10}$:
(a) $\varepsilon=0.5$, (b) $\varepsilon=0.05$, and (c) $\varepsilon=0.005$.}
\label{Figuresup2}
\end{figure}

From Tables \ref{t1}-\ref{t2} and Figure \ref{Figuresup2}, we can draw the following conclusions
\begin{itemize}
\item The MTI-FP method \eqref{MTIu}-\eqref{init00} achieves spectral accuracy in space when the solution is smooth,
which is uniformly for $\varepsilon\in(0,1]$ (cf. Table \ref{t1}). It is second order accurate in time when $\tau\lesssim \varepsilon^2$ (cf. Table \ref{t2} upper triangle), and is first order accurate in time when $\varepsilon \lesssim \sqrt{\tau}$
(cf. Table \ref{t2} lower triangle).  The method obtains uniformly first order convergent rate in time for $\varepsilon\in(0,1]$
(cf. Table \ref{t2} last row). All these numerical results confirm our error bounds in Theorem \ref{tm4.1} and demonstrate 
that they are sharp.

\item The multiscale interpolation \eqref{MTIint1}-\eqref{MTIint2} is uniformly first-order accurate in time
with respect to $\varepsilon\in(0,1]$. The results confirm our error bounds in Theorem \ref{tm2}.
\end{itemize}

To illustrate the performance of the proposed MTI-FP method \eqref{MTIu}-\eqref{init00} 
in high dimensions, we take $d = 2$, $\lambda = 1$,
 $\phi_1(x,y) = \text{exp}(-x^2 - (y+2)^2) + \text{exp}(-x^2 - (y-2)^2)$ and 
 $\phi_2(x,y) = \text{exp}(-x^2 - y^2)$ in \eqref{1.1}. The bounded computational domain 
 is taken as $\Omega = (-20,20) \times (-20,20)$. Figure~\ref{Figure2D} shows time evolution
 of the solution $u$ obtained numerically for 
 $\varepsilon=1$ and $\varepsilon=0.01$.
 
\begin{figure}[htbp]
\centering
\includegraphics[width=0.78\textwidth]{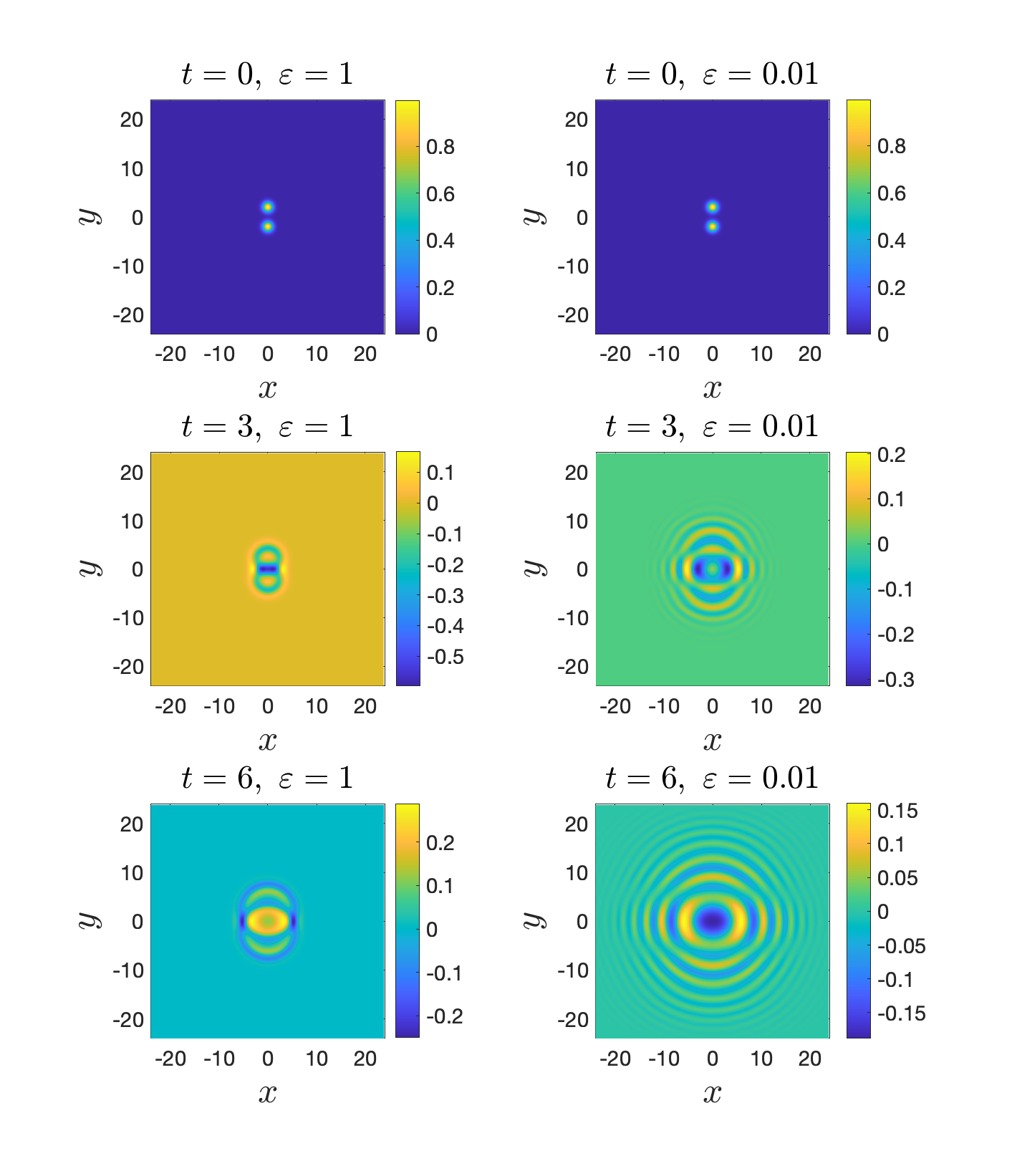}
\caption{Dynamics of the solution $u$ of the NKGE \eqref{1.1} in two dimensions for $\varepsilon= 1$ (left column) and $\varepsilon = 0.01$ (right column).}
\label{Figure2D}
\end{figure}

\begin{figure}[htbp]
\centering
\includegraphics[width=0.74
\textwidth]{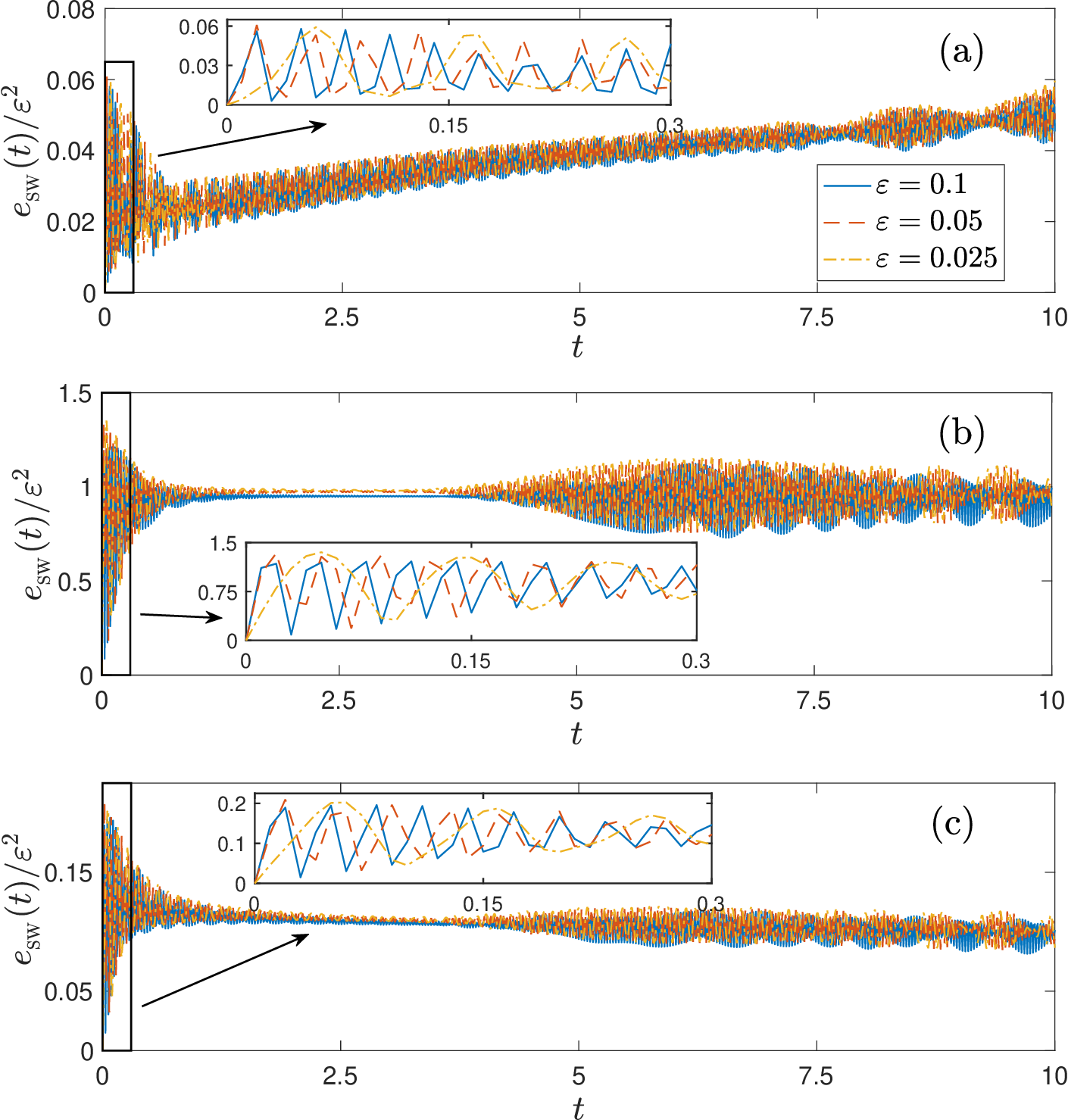}
\caption{Convergence of the NKGE \eqref{NKGE1d} to its limiting model NLSW \eqref{NLSW1} with different $\gamma$ in 1D:
(a) $\gamma(x)\equiv 0$, (b) $\gamma(x)$ is taken as the well-prepared initial data \eqref{2.12}, 
and (c) $\gamma(x) =\frac{3}{2} \lambda i |v_0(x)|^2 v_0(x)$ with $v_0(x)$ given in \eqref{initv23}.}
\label{Figurelim}
\end{figure}

\begin{figure}[htbp]
\centering
\includegraphics[width=0.68
\textwidth]{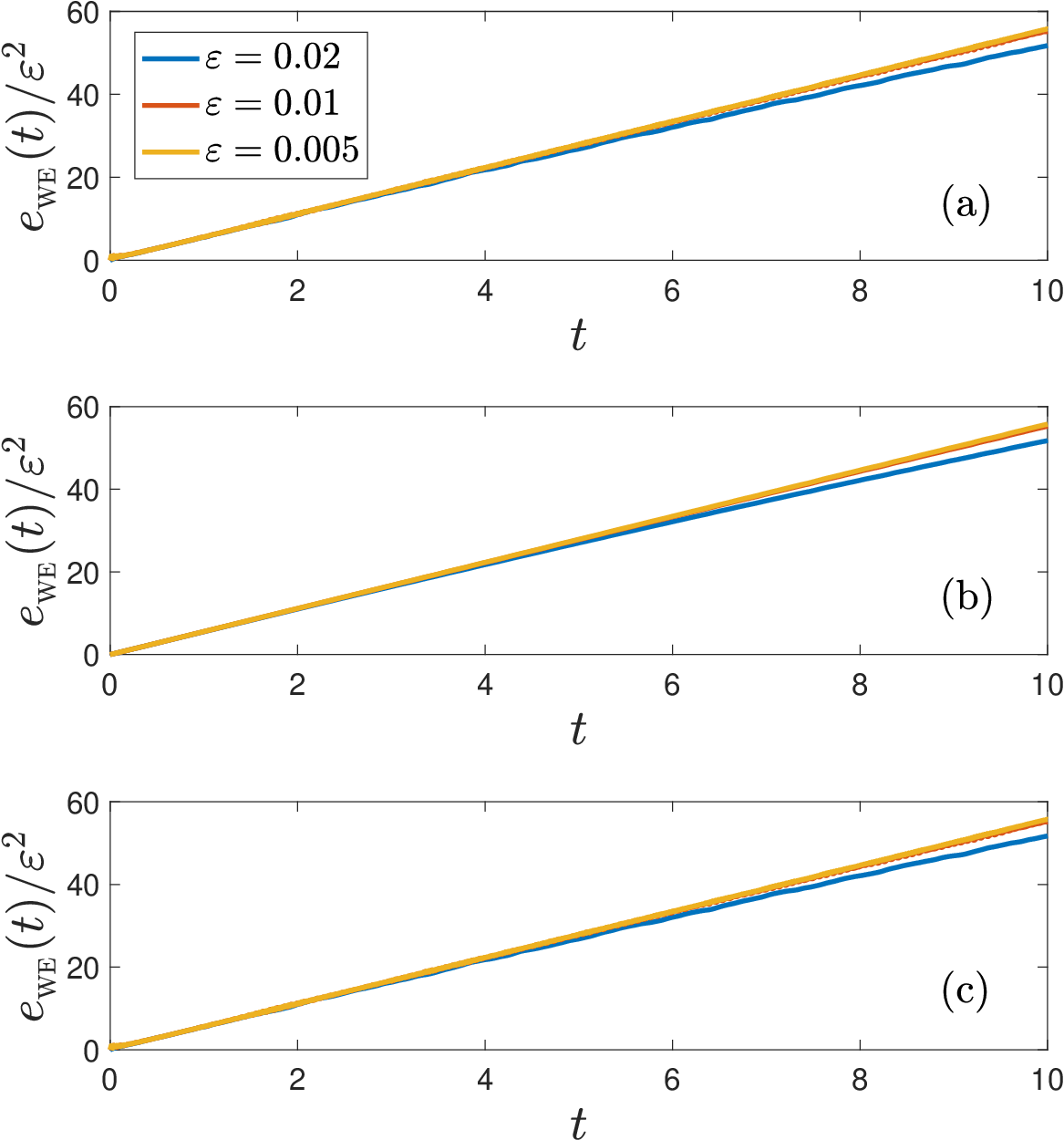}
\caption{Convergence of the NLSW \eqref{NLSW1} to its limiting model NLSE \eqref{NLSE1} with different $\gamma$ in 1D:
(a) $\gamma(x)\equiv 0$, (b) $\gamma(x)$ is taken as the well-prepared initial data \eqref{2.12}, 
and (c) $\gamma(x) =\frac{3}{2} \lambda i |v_0(x)|^2 v_0(x)$ with $v_0(x)$ given in \eqref{initv23}.}
\label{Figurelim2}
\end{figure}

\subsection{Convergent rates from NKGE to its limiting models}
Now we apply the proposed MTI-FP method \eqref{MTIu}-\eqref{init00} to study numerically 
convergent rates from the NKGE \eqref{1.1} to its limiting model NLSW \eqref{NLSW1} with different $\gamma$ and 
convergent rates from the NLSW \eqref{NLSW1} with different $\gamma$ to its limiting model NLSE \eqref{NLSE1}.   
The NLSW \eqref{NLSW1} is solved numerically by the exponential wave integrator Fourier pseudospectral (EWI-FP)
method \citep{bao2024optimal,hochbruck2010exponential} and the NLSE \eqref{NLSE1} is solved numerically by the time-splitting Fourier pseudospectral (TSSP) method \citep{bao2023improved,bao2024errors}.

Denote $v_{_{\rm SW}}:=v_{_{\rm SW}}(\bx,t)$ the solution of the NLSW \eqref{NLSW1} 
and $v_{_{\rm SE}}:=v_{_{\rm SE}}(\bx,t)$
the solution of the NLSE \eqref{NLSE1}. Define the error functions 
\begin{align*}
&e_{_{\rm SW}}(t) = \left\| u(\cdot,t) - \left[e^{it / \varepsilon^2} v_{_{\rm SW}}(\cdot,t) + e^{-it / \varepsilon^2} \overline{v_{_{\rm SW}}(\cdot,t)}\right] \right\|_{H^1},\\
&e_{_{\rm WE}}(t)=\left\|v_{_{\rm SW}}(\cdot,t)-v_{_{\rm SE}}(\cdot,t)\right\|_{H^1}, \qquad t\ge0.
\end{align*}
Figure \ref{Figurelim} plots the errors $e_{_{\rm SW}}(t)$ in one dimension (1D), i.e. we take $d=1$, $\lambda = 1$ and 
the initial data \eqref{5.1} in \eqref{1.1}, and $\Omega = (-16,16)$ in \eqref{NKGE1d} with a very fine mesh $h = 1/32$ and a time step $\tau = 1 \times 10^{-6}$, for different choices of $\gamma(x)$ in \eqref{NLSW1}.
Figure \ref{Figurelim2} depicts the errors $e_{_{\rm WE}}(t)$ under a similar set-up. We also carry out numerical study in two dimensions and three dimensions, and the conclusion is similar and thus the details are omitted here for brevity.

From Figures \ref{Figurelim} \& \ref{Figurelim2}, we can draw the following conclusions: 
(i) The NKGE \eqref{NKGE1d} converges to its limiting model NLSW \eqref{NLSW1} quadratically with respect to $\varepsilon$ (and globally in time) for different choices of $\gamma(\bx)$, i.e.
\begin{align}
    e_{_{\rm SW}}(t) \le C_\gamma  \varepsilon^2, \quad t \ge 0, \label{5.6}
\end{align}
where the positive constant $C_\gamma$ depends on $\gamma$, but is independent of $\varepsilon$ and time $t$.
In general, when $\gamma(\bx)\equiv 0$, the constant $C_\gamma$ becomes the minimum among different choices of $\gamma$
(cf. Fig. 4).
(ii) The NLSW \eqref{NLSW1} converges to its limiting model NLSE \eqref{NLSE1}
quadratically with respect to $\varepsilon$ for different choices of $\gamma(\bx)$, i.e.,
\begin{align}
    e_{_{\rm WE}}(t) \le (C_1+C_2 t)  \varepsilon^2, \quad t \ge 0, \label{5.7}
\end{align}
where the two positive constants $C_1$ and $C_2$ are independent of $\varepsilon$ and time $t$ (cf. Fig. 5).


\section{Conclusion} \label{sec7}
A uniformly accurate multiscale time integrator Fourier pseudospectral (MTI-FP) method was proposed and analyzed for solving the nonlinear Klein-Gordon equation (NKGE) in the nonrelativistic limit regime with a dimensionless parameter $0 < \varepsilon \le 1$. The method was designed based on two ingredients: (i) a multiscale decomposition by frequency of the NKGE in each time interval with simplified transmission conditions, and  (ii) an exponential wave integrator for temporal discretizaton and a Fourier pseudospectral method for spatial discretization. We established optimal 
error bounds for the MTI-FP method, which is uniformly accurate with respect to $\varepsilon\in (0,1]$.
In addition, a multiscale interpolation in time was presented for obtaining a uniformly accurate approximation of the solution at any time $t\ge0$ by adopting a linear interpolation of the micro variables within the multiscale decomposition in each time interval. Numerical results were reported to confirm our error bounds and to demonstrate their sharpness as well as to show quadratically convergent rates of the NKGE to its different limiting models in the nonrelativistic regime.

\section*{Funding}
The authors thank the anonymous reviewers for their valuable suggestions. This work was partially supported by the Ministry of Education of Singapore under its AcRF Tier 1 funding A-8003584-00-00.


\end{document}